\documentclass[12pt]{amsart}

\usepackage[hang,flushmargin]{footmisc}
\usepackage{bbm}
\usepackage[usenames,dvipsnames]{xcolor}
\usepackage{soul}
\usepackage{graphicx,rotating}
\usepackage{subfigure}
\usepackage[font=footnotesize,labelfont=footnotesize]{caption}
\usepackage[utf8]{inputenc}
\usepackage[toc]{appendix}
\usepackage{twoopt}
\usepackage{float}
\usepackage{comment}

\allowdisplaybreaks

\newcommand\rurl[1]{%
  \href{https://#1}{\nolinkurl{#1}}%
}

\usepackage[loadonly]{enumitem}
\usepackage{amssymb}
\usepackage{amsthm}
\usepackage{amsmath,a4wide}
\numberwithin{equation}{section}
\usepackage{hyperref}
\usepackage{mathtools}
\usepackage{amsfonts}
\usepackage[foot]{amsaddr}

\usepackage{marvosym}
\usepackage{verbatim}
\usepackage{marginnote}
\usepackage{enumerate}

\usepackage{graphicx}
\graphicspath{{graphics/}}

\usepackage{enumerate}
\mathtoolsset{showonlyrefs}

\theoremstyle{plain}
\newtheorem{theorem}{Theorem}[section]
\newtheorem*{theorem*}{Theorem}
\theoremstyle{plain}
\newtheorem{corollary}[theorem]{Corollary}
\newtheorem*{corollary*}{Corollary}
\theoremstyle{plain}
\newtheorem{lemma}[theorem]{Lemma}
\newtheorem*{lemma*}{Lemma}
\theoremstyle{plain}
\newtheorem{proposition}[theorem]{Proposition}
\newtheorem*{proposition*}{Proposition}
\theoremstyle{definition}

\theoremstyle{remark}
\newtheorem*{remark}{Remark}
\theoremstyle{remark}

\theoremstyle{remark}

\theoremstyle{definition}
\newtheorem{example}[theorem]{Example}
\newtheorem*{example*}{Example}

\theoremstyle{plain}

\theoremstyle{definition}

\providecommand{\norm}[1]{\left\lVert #1 \right\rVert}
\providecommand{\abs}[1]{\left\lvert #1 \right\rvert}

\newcommand{\R}{\mathbb{R}}
\newcommand{\C}{\mathbb{C}}

\newcommand{\Z}{\mathbb{Z}}
\newcommand{\N}{\mathbb{N}}

\newcommand{\G}{\mathcal{G}}

\newcommand{\indicator}{\raisebox{2pt}{$\chi$}}

\newcommand{\V}{\mathcal{V}}
\renewcommand{\d}{\mu}  %Used submultiplciity 
\newcommand{\U}{\mathcal{U}}
\renewcommand{\P}{\mathcal{P}}
\newcommand{\PP}{\mathbf{P}}
\newcommand{\VV}{\mathbf{V}}
\newcommand{\y}{\xi}
\newcommand{\MM}{\mathbf{M}}
\newcommand{\mg}{\mathfrak{mg}}

\DeclareMathOperator{\supp}{supp}
\DeclareMathOperator{\spanSpace}{span}
\DeclareMathOperator{\id}{id}

\newcommand{\sis}{shift-invariant space}
\newcommand{\mgg}{\mg (X,\mu _X)}
\definecolor{light-gray}{gray}{0.93}
\sethlcolor{light-gray}

\makeatletter
\newcommand\thankssymb[1]{\textsuperscript{\@fnsymbol{#1}}}

\begin{document}

\pagestyle{myheadings}

\title[Sampling with periodic exponential B-splines]{Sampling theorems with derivatives in shift-invariant spaces generated by periodic exponential B-splines}

\address{\thankssymb{1}Faculty of Mathematics, University of Vienna, Oskar-Morgenstern-Platz 1, 1090 Vienna, Austria}
%\address{\thankssymb{2}Acoustics Research Institute, Austrian Academy of Sciences, Wohllebengasse 12-14, 1040 Vienna, Austria}

\author[\qquad \qquad \qquad \qquad \quad Karlheinz \ Gr\"ochenig]{Karlheinz Gr\"ochenig \thankssymb{1} 
}
\email{karlheinz.groechenig@univie.ac.at}

\author[Irina Shafkulovska\qquad\qquad\qquad ]{Irina
  Shafkulovska\thankssymb{1}}
  %\textsuperscript{,}\thankssymb{2} }
\email{irina.shafkulovska@univie.ac.at}
\thanks{K. Gr\"ochenig was supported by the Austrian Science Fund (FWF) project P31887-N32. I. Shafkulovska was funded by the Austrian Science Fund (FWF) project P33217 and by the Austrian Academy of Sciences through an internship at the Acoustics Research Institute.}
 %\vspace{2em}
\date{}

\begin{abstract}
We derive sufficient conditions for sampling with derivatives in shift-invariant spaces generated by a periodic exponential B-spline. The sufficient conditions are expressed with a new notion of measuring the gap between consecutive samples. These conditions are near optimal, and, in particular, they imply the existence of sampling sets with  lower Beurling density arbitrarily close to the necessary density.  
\end{abstract}

\maketitle

\section{Introduction}

By a sampling problem, we understand the question of whether and how a function in a given space can be recovered from discrete information, which is usually given by point evaluations (samples). The given space is an a priori signal model, and for a century this model has been the Paley-Wiener space of bandlimited functions. The significance of the
Paley-Wiener space and the associated   Shannon-Whittaker-Kotelnikov sampling theorem cannot be overestimated; 
it is at the origin of the  information theory by  Shannon, it  is the technical basis for analog/digital conversion in signal processing, and it informs a whole branch of complex analysis. From the point of view of mathematics, the decisive results are the necessary density conditions of Landau~\cite{Landau1967} and the sufficient conditions for sampling by Beurling~\cite{BeurlingCollected1989}. Their work remains the point of reference for the exploration of sampling problems in analysis and signal processing. See~\cite{AldroubiGroechenig2000,BF01,seip04} for some representative expositions of modern sampling theory.   

In this paper, we study sampling in a \sis\  generated by a periodic exponential $B$-spline (PEB-spline). Given a PEB-spline, or more
generally a generator  $\varphi $ subject to some mild decay
conditions, the \sis\ generated by $\varphi $ in $L^p(\R)$ is defined to be 
\begin{equation}
V^p(\varphi)\coloneqq\Big\lbrace f\in L^p(\R) : f(x) = \sum\limits_{\ell\in\Z} c_\ell \varphi (x-\ell ) ,\ (c_\ell)_{\ell\in\Z}\in\ell^p(\Z) \Big \rbrace \subseteq L^p(\R).
\end{equation}
If $\varphi (x) = \tfrac{\sin \pi x}{\pi x}$, then $V^2(\varphi)= \{ f\in L^2(\R): \supp \hat{f} \subseteq [-1/2,1/2]\}$ is the Paley-Wiener space $PW^2(\R)$ of band-limited functions.

Shift-invariant spaces form a natural class of  \emph{signal models} in sampling theory, as they  generalize the Paley-Wiener space and are useful for practical and numerical reasons. The generator $\varphi $ can be considered a parameter of the signal model that can be tuned to the specifics of the data acquisition. For instance, if $\varphi $ is compactly supported, then every point evaluation can be computed as a finite sum, and a sample $f(x)$ affects only the coefficients $c_k$ in a neighbourhood of $x$. Additional local properties of $f\in V^p(\varphi)$, such as differentiability and analyticity, can be implemented by imposing them on the generator $\varphi$. By contrast, the Paley-Wiener space consists of entire functions, and every sample of $f$ has a long-range effect on the values of $f$. Shift-invariant spaces are central in approximation theory~\cite{BVR94,JetterZhou1995,Kyriazis1995, Ron2001} and are increasingly used in signal processing. 

For the mathematical formulation of the sampling problem, we assume that the signal space $V^p(\varphi )$ is known and that $f\in V^p(\varphi )$ is sampled on a set $X$ together with  the first derivatives $f^{(s)}(x), s=0, \dots , \mu (x),$ for some multiplicity function $\mu $. Whereas the standard sampling problem only takes samples on a discrete subset of $\R $, the inclusion of derivatives has become important in several applications. First derivatives model trends, and higher derivatives indicate convexity properties. In higher dimensional spaces, i.e., for images or physical fields, one speaks of ``gradient-augmented measurements''~\cite{AdcockSui2019}, in event-based sampling ``information is transmitted  only when a significant change in the signal occurs, justifying the acquisition of a new sample''~\cite{Pawlowski10}, in other words, the value of the first derivative determines the position of the next sample. 

For the Paley-Wiener space, sampling with derivatives was studied early on~\cite{Gro92b,rawn89,Razafinjatovo1995}, the problem was taken up again recently with an emphasis on \sis s in~\cite{AGH17,GSelvan23,GRS20,SelvanRiya2022}.

The main objective  is to recover a function from its samples and derivatives on a
set $X= (x_j)_{j\in \Z}$. We indicate the number of derivatives at every point with a \emph{multiplicity function}
$\mu_X:X\longrightarrow\lbrace 0,\dots,S\rbrace$. The goal is to derive  a sampling  inequality of type
\begin{equation} \label{in11}
A_p \norm{f}_p^p \leq  \sum\limits_{x\in
  X}\sum\limits_{s=0}^{\mu_X(x)}\abs{f^{(s)}(x)}^p \leq B_p
\norm{f}_p^p \qquad \text{for all
}f\in  V^p(\varphi). 
\end{equation}
Note that the number of derivatives depends on the sampling point $x$. If \eqref{in11} is satisfied, we call $(X,\mu _X)$ a sampling set
(or a set of stable sampling) for $V^p(\varphi )$. Once a sampling inequality is available, one can use standard reconstruction algorithms from frame theory to reconstruction $f$ from $(X,\mu _X)$. 

Our final choice is a class of \emph{generators} for the \sis . We suggest the use of PEB-splines as suitable and practical generators. These have compact support, they are well understood in approximation theory~\cite{deBoor78,Schumaker2007}
and in numerics. In sampling theory, they lead to
$\mathcal{O}(N)$ algorithms, when $N$ is the number of available data~\cite{GS03}. Most importantly and in contrast to other classes of generators, for PEB-spline generators one can prove optimal results. This is precisely the objective of this paper.

Before proving anything at all, one should understand what is possible. How much information is necessary to recover a function from its samples? This question is answered by a necessary
density condition. For \sis s the relevant quantity is the weighted Beurling density defined as
\begin{equation}
D^{-}(X,\mu_X):=\liminf\limits_{R\rightarrow\infty}\inf\limits_{y\in\R}\frac{1}{2R}\sum\limits_{x\in
  X\cap  [y-R, y+R]} (1+\mu_X(x)). 
\end{equation}
If $\mu \equiv 0$, then $D^-(X,\mu _X) = D^-(X)$ is just the usual lower Beurling density. The weighted Beurling density counts the average number of data per unit interval. The necessary density condition, first proved by Landau~\cite{Landau67} for bandlimited functions in full generality (with arbitrary spectra and in higher dimensions), turns out to be a universal information-theoretic bound for sampling. In the context of sampling with derivatives in \sis s  it can be expressed as follows~\cite[Prop.~3.7.]{GroechenigEtAl2019}. 

\begin{theorem} Assume that $\mu _X$ is bounded by $M$ and  $\varphi ,  \varphi ' , \dots ,  \varphi ^{(M)}$ satisfies a mild decay condition, e.g., $|\varphi ^s (x)| \leq C(1+|x|)^{-1-\varepsilon}$ for some $\varepsilon >0$.

If $(X,\mu)$ is a sampling set for $V^p(\varphi )$, then
\begin{equation} \label{eq:ce5}
  D^-(X,\mu _X) \geq 1 \, .  
\end{equation}
\end{theorem}

The  problem of sufficient conditions that guarantee the reconstruction of $f\in V^p(\varphi )$ from its samples $(X,\mu_X)$  is more interesting and also more difficult. Under minimal conditions on the smoothness and decay of $\varphi $ one can always show  with elementary tools that any sufficiently dense set is sampling~\cite{AF98}. 
For some types of generators, the density condition~\eqref{eq:ce5} is (almost) a
characterization. If $\varphi $ is either the cardinal sine
$\frac{\sin \pi x}{\pi x}$, whence $V^2(\varphi ) = PW$, or the Gaussian or some related function, then $D^-(X,\mu )>1$ is a sufficient condition for
sampling~\cite{BeurlingCollected1989,GroechenigEtAl2019}. For other generators, \eqref{eq:ce5} may not provide a characterization, in particular, if $\varphi $ has compact support, then the maximal gap between sampling points is limited by the size of $\supp \varphi $, whereas for given Beurling density large gaps are permitted. Therefore one often uses different notions to measure the density of a set. 

\vspace{3mm}

\textbf{Contributions.} 
For PEB-splines as generators,  we will derive a set of conditions that allow the construction of sampling sets arbitrarily close to the necessary condition.

Our first theorem is somewhat technical because  the sufficient conditions are expressed in terms of local Schoenberg-Whitney conditions that are
well-known in spline theory. These conditions come close to a characterization of sampling sets with derivatives but are not complete. We postpone the precise formulation to Section~\ref{sec:main1}, and state the most important consequence.

\begin{theorem}\label{thm:existencei}
Let $\varphi $ be a PEB-spline of order $m$ with support in $[0,m]$. For every $\varepsilon >0$  and given multiplicity sequence  
$\mu = (\mu _j)_{j\in \Z}$ with values in $\{0,\dots m-1\}$, there exists a set $X= (x_j)_{j\in \Z}$, such that
$(X,\mu_X)$ is sampling for $V^p(\varphi)$, $\mu_X(x_j)=\mu_j$ for all $j\in\Z$, and $D^-(X,\mu_X) < 1+\varepsilon $.
\end{theorem}

Our second result uses the  maximum gap between consecutive samples, also called the mesh width or the covering density. The maximum gap of a set $X = (x_j)_{j\in \Z}\subseteq \R $ 
consisting of consecutive samples  is defined as
\begin{equation}
\mg(X)= \sup_{j\in\Z} (x_{j+1}-x_{j}) \, ,
\end{equation}
and is used frequently to formulate conditions for sampling, e.g., in~\cite{AF98}. 
Our new contribution is a weighted maximum gap that takes into consideration the number of derivatives at each point. It is defined as
\begin{equation}
\mg(X,\mu_X)=\max\left\lbrace  \sup\limits_{j\in\Z} \frac{x_{j+1}-x_{j}}{1+\mu_X(x_j)}\, ,\, \sup\limits_{j\in\Z} \frac{x_{j+1}-x_{j}}{1+\mu_X(x_{j+1})}\right\rbrace 
\end{equation}
The rationale for this definition is that several derivatives at a point $x$, say $f^{(s)}(x), s= 0, \dots , \mu _X (x)$, amount to $\mu _X(x)+1$ data, so the next data point should be at distance $<\mu _X(x)+1$ from $x$, in other words,
$\frac{x_{j+1}-x_j}{\mu _X(x)+1}<1$. The precise formulation is as follows.
\begin{theorem}[{\hyperref[thm:max_gap]{Maximum Gap Theorem}}]
Let $\varphi$ be a PEB-spline of order $m$
and let $X\subseteq\R$ be a separated set with multiplicity function $\mu_X:X\to \lbrace 0,\dots, m-2\}$. If the weighted maximum gap of $(X, \mu_X)$ satisfies
\begin{equation} \label{opt?}
\mg(X,\mu_X)  <1,
\end{equation}
then $(X, \mu_X)$ is a sampling set for $V^p(\varphi)$.
\end{theorem}
In contrast to the more technical theorem, the maximum gap condition
is easy to understand and to verify.

As a special case, we mention sampling with the same number of derivatives at every point. 

\begin{corollary} \label{allequalintro}
Let $\varphi$ be a PEB-spline of order $m\geq 2$
and let $X\subseteq\R$ be a separated set. Assume that $\mu _X$ is constant, $\mu _X(x) = s < m-1$ for all $x\in X$. If 
\begin{equation}
\mg(X)< s+1,
\end{equation}
then $(X, \mu_X)$ is a sampling set for $V^p(\varphi)$, $1\leq p\leq\infty$.

In particular, a lattice  $X=\alpha \Z $ is a 
sampling set for $V^p(\varphi), 1\leq p\leq \infty $, if and only if $\alpha < s+1$. 
\end{corollary}

When using $\mg (X,\mu _X)$ as the relevant density and  PEB-splines as generators, then the weighted gap condition~\eqref{opt?} is best possible in the following sense. One can show that for $a>1$ and given multiplicity sequence $(\mu _j)_{j\in \Z }$ there exist sets $X= (x_j)$, such that $\mg (X,\mu ) = a$, but $(X,\mu )$ is not a sampling set. 

It is easy to see that  $D^-(X,\mu _X) \geq \mg (X,\mu _X)^{-1}$. This explains why the maximum gap  condition is often problematic. So far there is only a handful of generators~\cite{AldroubiGroechenig2000,SR17a, GroechenigEtAl2017,GroechenigEtAl2019} for which the condition $\mg (X) <1$ or $\mg (X, \mu _X) < 1$ has been proved to be sufficient for sampling. In a surprisingly large number of publications, a
sufficient condition $\mg (X) \leq \delta _0$  has been derived for unknown or small $\delta_0$. Since the comparison to the necessary condition is missing, more research is required to understand the relevance of such results.

In the last section, we exploit a connection between sampling in \sis s and the theory of Gabor frames and derive some new results about Gabor frames with PEB windows. 

\vspace{3mm}

\textbf{Methods.} 
Following~\cite{AldroubiGroechenig2000} we partition $\R $ into intervals of finite length and then study the local problem. Since a PEB-spline of order $m\in\N$ has support in the  interval $[0,m]$, the restriction of a function $f$ in $V^p(\varphi )$ to an interval $[a,b), a,b\in \Z$, sees only the finite sum 
$$
f(x) = \sum _{\ell = a-m+1}^{b-1} c_{\ell} \varphi (x-\ell ) \qquad \text{ for } x\in [a,b) \, .
$$
This is then a finite-dimensional problem, and the intervening matrix is the \emph{collocation matrix} associated to the PEB-spline $\varphi 
$. Its invertibility is well understood and can be expressed with the help of the Schoenberg-Whitney conditions on the sampling points in
$[a,b]$. The characterization of invertible collocation matrices is one of the cornerstones of spline theory~\cite{Schumaker2007}, and our proofs pay homage to this fundamental result. We
develop two approaches centred around the collocation matrix. In one approach we force uniform bounds on the collocation matrix and then glue together the local reconstructions. In an alternative approach, we use Beurling's technique of weak limits to arrive at a sampling theorem. 

\vspace{3mm}

The paper is organized as follows. Section \ref{sec:SI-spaces} introduces shift-invariant spaces and sampling sets and prepares some technical tools for sampling  with derivatives. 
Section \hyperref[sec:ECC]{3} is dedicated to the basic theory of ECC-systems and B-splines with a special focus on collocation matrices for sampling with derivatives. 
The first main result and the existence of sampling sets near the optimal density are proven in Section \ref{sec:main1}. Section \ref{sec:weak_limits} is devoted to Beurling's technique adapted to sampling in shift-invariant spaces. 
With this at hand, we prove the main sampling with the maximum gap condition in Section \ref{sec:main2}. 
We conclude with Section \ref{sec:discussion} with a discussion of the two methods and Section \ref{sec:gabor}, where we derive the implications of the two sampling theorems for Gabor systems. A technical result is postponed to \hyperref[sec:appendix]{the appendix}.

\section{Shift-invariant spaces and sampling}\label{sec:SI-spaces}
We begin with properties of \sis s and various formulations of the  sampling problem. 
\subsection{(Vector-valued) shift-invariant spaces}
Let us consider a shift-invariant space $V^p(\varphi)$. Since each element $f\in V^p(\varphi)$ is represented as $f = \sum_{\ell\in\Z}c_\ell T_\ell\varphi$ for some sequence $\left(c_\ell\right)_{\ell\in\Z}\in\ell^p(\Z)$, the desirable generators are those which come with the norm equivalence
\begin{equation}\label{eq:stable_def}
A_p \norm{c}_{\ell^p} \leq \Big\lVert\,\sum_{\ell\in\Z}c_\ell
T_\ell\varphi\,\Big\rVert_{L^p}\leq B_p \norm{c}_{\ell^p}, \qquad c\in\ell^p(\Z),
\end{equation}
where the constants $A_p,B_p>0$ depend only on $p$. If this holds, we say that $\varphi$ has \textit{$p$-stable integer translates (shifts)} and denote \eqref{eq:stable_def} as 
$\norm{f}_{L^p} \asymp \norm{c}_{\ell^p} $. 
If norm equivalence holds for all $p\in[1,\infty]$, we omit the reference to $p$ and speak of \textit{stable integer translates} and \textit{a stable generator}. 

The following theorem summarizes some well-known conditions for
stability. 
\begin{theorem}\label{thm:stable_shifts} Let $\varphi:\R\to\C$ be a bounded, compactly supported function. Then the following statements are equivalent:
\begin{enumerate}[(i)]
\item The generator $\varphi$ has stable translates.
\item The generator $\varphi$ has stable $\infty$-translates.
\item The translates $\lbrace T_\ell\,\varphi : \ell\in\Z\rbrace$ are $\ell^\infty$-independent, i.e., 
\begin{equation}
\sum\limits_{\ell\in\Z} c_\ell\,T_\ell\, \varphi \not\equiv 0 \qquad\text{ for all }c\in\ell^\infty(\Z)\setminus\lbrace 0\rbrace.
\end{equation}
\item The Fourier transform $\widehat{\varphi}$ does not have a real $1$-periodic zero, i.e.,
\begin{equation}
0< \sum\limits_{\ell\in\Z} \abs{\hat{\varphi}(\omega+\ell)}^2 \qquad\text{ for all }\omega\in\R.
\end{equation}
\end{enumerate}
\end{theorem}
We refer to Ron's survey article~\cite{Ron2001}. The equivalences follow from the results in \cite[Thm.~3.5]{JiaMicchelli1991} and \cite[Thm.~29]{Ron2001}. For further references, the reader can consult \cite[Eq.~(9)]{Mallat1989},
\cite[Thm.~3.3]{JiaMicchelli1991}. For sampling in shift-invariant spaces the  stability condition is explained and used in \cite{AldroubiGroechenig2000,
  AldroubiGroechenig2001} and \cite[Sec.~2]{GroechenigEtAl2017}.

For the special case of a bounded, compactly supported  generator $\varphi$ above with $\mathrm{supp}(\varphi) \subseteq [-L,L]$, the important lower inequality is almost trivial: 
\begin{equation}
\norm{f}_{L^p}^p \leq \int_{\R}\sum\limits_{\ell\in\Z} \abs{c_\ell
T_\ell \varphi(t)}^p\, dt  \leq \sum\limits_{k\in\Z}\int\limits_{K}^{K+1}\sum\limits_{\ell = K-2L+1}^{L} \abs{c_\ell}^p \norm{\varphi}_{\infty}\, dt = 4L \norm{\varphi}_\infty \norm{c}_p^p, 
\end{equation}
if $p<\infty$, with the analogous estimate for $p=\infty$. 

Following \cite[Sec.~2]{GroechenigEtAl2019}, we further consider vector-valued shift-invariant spaces
\begin{equation}
V^p(\Phi) :=\left\lbrace\sum\limits_{\ell\in\Z}c_\ell \, T_\ell\,\Phi : c\in\ell^p(\Z)\right\rbrace \subseteq (L^p(\R))^{S+1}
\end{equation} 
generated by a vector $\Phi = \left(\Phi^0,\dots, \Phi^S\right)^t\in\left(L^1(\R)\right)^{S+1}$ with norm
\begin{equation}
\norm{F}_p =\left(\sum\limits_{s=0}^S \norm{F^s}_{L^p}^p\right)^{1/p},
\end{equation}
with the usual modification for $p=\infty$. We will assume that $\Phi$ has stable integer shifts, i.e., 
\begin{equation}\label{c1}
\Big\lVert\,\sum\limits_{\ell\in\Z}c_\ell \, T_\ell\,\Phi\,\Big\rVert_p \asymp \norm{c}_p \qquad\text{ for all }c\in \ell^p(\Z).
\end{equation}

\subsection{Sampling sets}\label{sec:sampling_sets}
To accommodate derivatives and  multiplicities in sampling, we consider tuples of sets $\overrightarrow{X}=\left(X^0,\dots, X^S\right)$ with $X^s\subseteq\R$, $0\leq s\leq S$. Assume that the generator has stable integer shifts and that the shift-invariant space $V^p(\Phi)$  contains tuples of functions $F=(F_0,\dots, F^S)$ which are pointwise well-defined. In the following, $V^p(\Phi^s)$ contains only continuous or piecewise continuous functions with well-defined one-sided limits at the finitely many discontinuities. 
We say that $\overrightarrow{X}$ is a \emph{sampling set} for $V^p(\Phi)$, $p\in[1,\infty]$, if 
\begin{equation}\label{eq:sampling_vectors_def}
\sum\limits_{s=0}^S \sum\limits_{x\in X^s} \abs{F^s(x)}^p \asymp \norm{F}_p^p \qquad\text{for all }F\in V^p(\Phi),
\end{equation}
where the inequality constants depend only on $p$ and $\Phi$ (with the usual adaptation for $p=\infty$). 
A weaker notion is that of  a uniqueness set. We call $\overrightarrow{X}$ a \emph{uniqueness set} if the linear operator $F\mapsto (F^0|_{X^0},\dots, F^s|_{X^s})$ is injective on $V^p(\Phi)$, that is,
\begin{equation}
{F^s}\big|_{{X^s}} \equiv 0 \ \ \text{ for all } s\in\{ 0,\dots, S\} \qquad \Leftrightarrow \qquad F\equiv 0.
\end{equation}

The case $S=0$ is the standard case of sampling sets. As usual, the upper bound for sampling is easy to obtain in a general setting, whereas, for the lower bound, the specific features of the generator have to be exploited.

\begin{theorem}\label{thm:upper_bound_W_jumps}
Let $\varphi$ be a compactly supported, piecewise continuous function with finitely many jump discontinuities and assume that $\varphi$ has stable integer translates. Then 
\begin{enumerate}[(i)]
\item All functions in $V^p(\varphi)$ are piecewise continuous. Furthermore, the set of discontinuities is separated, consisting only of jump discontinuities. 
\item The sampling operator $f\to f|_{X}$ is bounded on $V^p(\varphi)$.
Here the samples at discontinuities are to be understood as samples of the right-sided limits.
\end{enumerate}
The analogous statement holds with left-sided limits.
\end{theorem}

\begin{proof}
\textit{(i)} A series $\sum_{\ell\in\Z}c_\ell T_\ell \varphi$ is locally finite, so it inherits the properties of the generator. 

Concerning \textit{(ii)}, we refer the reader to
\cite[Thm.~3.1(iii)]{AldroubiGroechenig2001} for the proof~\footnote{Note that
  the proof does not require the continuity of $\varphi $.}.
\end{proof}

To relate sampling in vector-valued spaces to sampling with derivatives with a  multiplicity function $\mu_X: X\to\{0,\dots, S\}$, 
we set $\Phi\coloneqq (\varphi,\varphi^{(1)},\dots,\varphi^{(S)})$ and 
\begin{equation}
X^s\coloneqq \{ x\in X : \mu_X(x)\geq s\},\qquad 1\leq s\leq S.
\end{equation} 
In this case, the inequality \eqref{eq:sampling_vectors_def} for vector-valued functions is the same as the sampling inequality 
\begin{equation} 
\sum\limits_{x\in X} \sum\limits_{s=0}^{\mu_X(x)}\abs{f^{(s)}(x)}^p
\asymp \norm{f}_p^p  \qquad \text{for all }f\in  V^p(\varphi).
\end{equation}

\section{Some Spline Theory}\label{sec:ECCa}
In this section, we introduce ECC-systems and Chebyshev B-splines and cite and prove  statements about collocation matrices arising in the sampling problem. We follow  the terminology in \cite{Schumaker2007}. 

\subsection{ECC-systems and Chebyshev B-splines}\label{sec:ECC}
We call a system of functions $u_\ell:[a,b]\longrightarrow\C$, $1\leq \ell\leq m$,  \textit{a Chebyshev system}, if the collocation matrix with entries 
\begin{equation}
\left(\, u_\ell\,(t_i)\,\right)_{1\leq i,\ell\leq m}
\end{equation}
has a strictly positive determinant for all $a\leq t_1<\dots <t_m\leq b$. It follows from the definition that a Chebyshev system is linearly independent. The span of $\left(u_1, \dots, u_m \right)$ is called a Chebyshev space. 

The definition of a Chebyshev system implies immediately the following property of the zero set. One of the main properties of these spaces is the number of solutions
of $f(x)=0$. 
\begin{lemma}\label{thm:Tzeros}
If $\left(u_1,\dots,u_m\right)$ is a Chebyshev system, then either $(c_\ell)_{1\leq \ell\leq m}=0$ or
\begin{equation}
\#\bigg\lbrace t\in[a,b]: \sum\limits_{\ell=1}^m c_\ell u_\ell(t)=0\bigg\rbrace\leq m-1.
\end{equation}
\end{lemma} 
Moving further towards splines, we choose positive weight functions $w_s\in C^{m-s}[a,b]$ on an interval $[a,b]\subseteq\R$ and define
\begin{align}
u_1(x)& = w_1(x),\\
u_2(x)& = w_1(x)\int\limits_a^x w_2(x_2)\,dx_2,\\
\vdots&\\
u_m(x)& = w_1(x)\int\limits_a^x w_2(x_2)\int\limits^{x_2}_a w_3(x_3) \dots \int\limits_a^{x_{m-1}} w_m(x_m)\, d x_m \dots\,dx_2.
\end{align}
These are functions in $\C^{m-1}[a,b]$ and form a so-called \emph{extended complete Chebyshev (ECC) system} on $[a,b]$. This means that for all $1\leq D\leq m$ and all $t_1\leq\dots\leq t_D\in[a,b]$ the \emph{collocation matrix of Hermite interpolation}, i.e., the matrix with entries 
\begin{equation}
\big(u_\ell^{(d_i)}\,(t_i)\big)_{1\leq i,\ell\leq D},
\end{equation}
where
\begin{equation}
d_i:=\max\left\lbrace \ell : t_i =\dots = t_{i-\ell}\right\rbrace,\qquad 1\leq i\leq D,
\end{equation}
has a positive determinant. This is the matrix that arises in the Hermite interpolation problem
\begin{equation}
f=\sum\limits_{\ell=1}^D c_\ell u_\ell,\quad  f^{(d_i)}(t_i) = \xi_i,\qquad 1\leq i\leq D.
\end{equation}
The associated ECC-space is given by $\U_m:=\spanSpace_\C{(u_1,\dots, u_m)}.$
As the simplest example (and the motivation), we consider $w_1(x)=1$ and $w_s = s-1$, $2\leq s\leq m$. In this case, we obtain the polynomials $u_s=(x-a)^{s-1}$, $1\leq s\leq m$, i.e., the ECC-space $\U_m = \mathcal{P}_{m}$ is the space of polynomials of degree at most $m-1$. 
One of the particularly useful properties of polynomial spaces is the self-containment 
with respect to the derivatives, i.e., 
\begin{equation}
p^{(s)} \in \mathcal{P}_{m-s},\qquad 0\leq s\leq m-1,\quad p\in\mathcal{P}_m.
\end{equation}
To preserve this property for ECC-systems, one replaces the usual derivatives with related differential operators. We define $D_0\,f=f$
and  
\begin{equation} 
D_s\, f:= \tfrac{d}{d x}\, \tfrac{f}{w_s},\quad L_s:= D_s D_{s-1}\dots D_0,\quad 0\leq s\leq  m.
\end{equation}
By Leibniz's rule, $L_s$ is a differential operator of order $s$ with
variable coefficients. 

Just as with the polynomials and the well-known B-splines, we can now consider Chebyshev splines. For an ECC-system $(u_1,\dots, u_m)$ on
$[a,b]$ and knots $a = y_1 <\dots <y_{D+m}= b$, a \emph{spline} with these knots is a  function $B\in C^{m-2}(\R)$ satisfying 
\begin{equation} 
B\vert_{(y_\ell,y_{\ell+1})}\in\U_m,\qquad 1\leq \ell\leq D+m-1.
\end{equation}
This condition implies that $B$ is piecewise $C^{m-1}$ and at the knots $B^{(m-1)}$ has jump discontinuities and only one-sided
derivatives exist. The integer $m$ is the order of the spline. 
In this case  we define $L_{m-1}B$ at the  knots $y_i$ to be the right-sided limit 
\begin{equation}
L_{m-1}B(y_i) \coloneqq \lim\limits_{x\searrow y_i}L_{m-1}B(x).
\end{equation}  

A major result in spline theory is the existence of  Chebyshev B-splines (CB-splines) associated to knots $a\leq y_1 <\dots < y_{m+D}\leq b$.  Precisely, the CB-splines  are the  splines $\big(B_m^\ell\big)_{1\leq \ell\leq D}$ with controlled support 
\begin{equation}
\supp(B_m^\ell) =[y_\ell,y_{\ell+m}], \qquad 1\leq \ell\leq D.
\end{equation}
The CB-splines $B_m^\ell $ are unique up to normalization. The traditional construction is based on divided differences and can be found in full detail in Schumaker's book \cite[Sec.~9.4]{Schumaker2007}.

We conclude by considering matrices associated to CB-splines. The definition of ECC-systems calls for the collocation matrix with the standard derivatives. We adapt this to the new differential operators $L_s$. Given $t_1\leq t_2\leq \dots\leq t_{D}$, let
\begin{equation}
d_i:=\max\left\lbrace \ell : t_i =\dots = t_{i-\ell}\right\rbrace\leq m-1,\qquad 1\leq i\leq D.
\end{equation} 
Let $(B_m^1,\dots, B_m^D)$ be CB-splines associated with the knots $y_1<y_2<\dots < y_{m+D}$. We consider the collocation matrix
\begin{equation}\label{eq:CB_L_Matrix_Def}
\big( L_{d_i}B_m^\ell(t_i)\big)_{1\leq i,\ell\leq D}.
\end{equation}
This is the matrix tied to the Hermite interpolation problem
\begin{equation}
f=\sum\limits_{\ell=1}^D c_\ell B^\ell_m,\quad L_{d_i} f(t_i) = \zeta _i,\qquad 1\leq i\leq D.
\end{equation}
We can determine exactly when this matrix is invertible \cite[Thm.~9.33]{Schumaker2007}, cf. \cite[Thm.~4.67]{Schumaker2007}.
\begin{theorem}[Interlacing property, Schoenberg-Whitney conditions]\label{thm:SW-cond_OV} 
The determinant of the collocation matrix in \eqref{eq:CB_L_Matrix_Def} is nonnegative. It is positive if and only if 
\begin{equation}
t_i\in \begin{cases}
(y_i,y_{i+m}), & \text{ if }  d_i<m-1 \\
[y_i,y_{i+m}), & \text{ if } d_i = m-1,
\end{cases}
\qquad 1\leq i\leq D.
\end{equation}
\end{theorem}  
The theorem says that it suffices for $t_i$ to be in the support of $B_m^i$. 
The invertibility of the collocation matrix is independent of the
normalization of the CB-splines, as the  normalization corresponds to multiplying the collocation matrix with an invertible diagonal matrix.

\begin{example}
As a small  example, we consider CB-splines of order 5. Let the knots be $0,1,\dots, 10$. Then the splines $B^0_5,\dots,B^5_5$ are supported on 
\begin{equation}
(0,5),\quad (1,6),\quad (2,7),\quad (3,8),\quad (4,9),\quad \text{ and }\quad (5,10),
\end{equation}
respectively. If we choose $t_1=t_2=0.5$, $t_3=4$, $t_4=t_5=t_6 = 7.7$, then the corresponding multiplicities are
\begin{equation}
d =\begin{pmatrix}
0 &1 & 0 & 0 & 1 & 2
\end{pmatrix}.
\end{equation}
The matrix is not invertible because $t_2 = 0.5$ is not in $(1, 6)$. Visually, the issue is clear:
\begin{align}
&\begin{pmatrix}
B_5^0(0.5) & B_5^1(0.5) & B_5^2(0.5) & B_5^3(0.5) & B_5^4(0.5) & B_5^5(0.5) \\[1ex]
L_1 B_5^0(0.5) & L_1 B_5^1(0.5) & L_1 B_5^2(0.5) & L_1 B_5^3(0.5) & L_1 B_5^4(0.5) & L_1 B_5^5(0.5) \\[1ex]
B_5^0(4) & B_5^1(4) & B_5^2(4) & B_5^3(4) & B_5^4(4) & B_5^5(4) \\[1ex]
B_5^0(7.7) & B_5^1(7.7) & B_5^2(7.7) & B_5^3(7.7) & B_5^4(7.7) & B_5^5(7.7) \\[1ex]
L_1 B_5^0(7.7) & L_1 B_5^1(7.7) & L_1 B_5^2(7.7) & L_1 B_5^3(7.7) & L_1 B_5^4(7.7) & L_1 B_5^5(7.7) \\[1ex]
L_2 B_5^0(7.7) & L_1 B_5^1(7.7) & L_2 B_5^2(7.7) & L_2 B_5^3(7.7) & L_2 B_5^4(7.7) & L_2 B_5^5(7.7)
\end{pmatrix} \\[2ex]
=& \begin{pmatrix}
B_5^0(0.5) & 0 & 0 & 0 & 0 & 0 \\[1ex]
L_1 B_5^0(0.5) & 0 & 0 & 0 & 0 & 0 \\[1ex]
B_5^0(4) & B_5^1(4) & B_5^2(4) & B_5^3(4) & 0 & 0 \\[1ex]
0 & 0 & 0 & B_5^3(7.7) & B_5^4(7.7) & B_5^5(7.7) \\[1ex]
0 & 0 & 0 & L_1 B_5^3(7.7) & L_1 B_5^4(7.7) & L_1 B_5^5(7.7) \\[1ex]
0 & 0 & 0 & L_2 B_5^3(7.7) & L_2 B_5^4(7.7) & L_2 B_5^5(7.7)
\end{pmatrix} \\
=& \begin{pmatrix}
\ * & 0 & 0 & 0 & 0 & 0\ \\[1ex]
\ * & 0 & 0 & 0 & 0 & 0\ \\[1ex]
\ * & * & * & * & 0 & 0\ \\[1ex]
\ 0 & 0 & 0 & * & * & *\ \\[1ex]
\ 0 & 0 & 0 & * & * & *\ \\[1ex]
\ 0 & 0 & 0 & * & * & *\
\end{pmatrix}
\end{align}
If instead we  sample at $t_1=t_2 = 1.5\in (1,5)$, $t_3=4$, $t_4=t_5=t_6 = 7.7$, all points are positioned correctly: $4$ is in $(2,7)$ and $7.7$ lies in $(3,8)$, $(4,9)$ and $(5,10)$. In this case, the collocation matrix is invertible.
\end{example}

\subsection{Periodic exponential B-splines}\label{subs:PEB-splines}
We have now gathered all the necessary background on CB-splines and continue with the special case of CB-splines which are shifts of a single CB-spline. For this, we require the following structure:
\begin{enumerate}[(S1)]
\item the knots are the integers $y_i= i-1$, $i\in\Z,$\footnote{The choice $y_i = i-1$ instead of $y_i=i$ is for the support condition \eqref{eq:ShiftInvariance_PEB} and consistency of Theorem \ref{thm:SW-cond_OV} with later applications.}
\item the weights $w_s$, $1\leq s\leq m$ are given by
\begin{equation}
w_s(x) = e^{\gamma_s x}r_s(x),\qquad \gamma_s\in\R, \qquad T_{\ell} r_s = r_s>0\quad \text{for all } \ell\in\Z.
\end{equation}
\end{enumerate}
CB-splines in this form are called \textit{periodic exponential
  B-splines} (PEB-splines) introduced in ~\cite{Ron88b, Ron88}. If $r_s\equiv 1$ for all $1\leq s\leq m$, then we refer to them as \textit{exponential B-splines} (EB-splines). The benefit of (S1) and (S2) is the fact that the CB-spline associated with the knots $y_\ell,\dots,y_{\ell+m}$ is (up to scaling) uniquely determined by any other of the B-splines, e.g., by $\varphi\coloneqq B_m^1$. 
It is clear that any shift $T_\ell \varphi$ is a Chebyshev spline with
\begin{equation}\label{eq:ShiftInvariance_PEB}
\supp(T_\ell \varphi) =[0+\ell,m+\ell],  \qquad \ell\in\Z.
\end{equation}
By the uniqueness, $(T_\ell \varphi)_{\ell\in\Z}$ differs up to positive factors $\left(\theta_\ell\right)_{\ell\in\Z}$ from the classically constructed CB-splines with divided differences $\left(B_m^\ell\right)_{\ell\in\Z}$ \cite[Sec.~9.4]{Schumaker2007}. The invertibility of the collocation matrix in Theorem \ref{thm:SW-cond_OV} is now tied to the integers. We take a closer look at the associated differential operators. For ease of reading, we define for all $1\leq s\leq m$ the functions
\begin{equation}
q_s(x)\coloneqq \frac{1}{r_s(x)}, \quad \text{so that} \quad w_s^{-1}(x) = e^{-\gamma_s x}q_s(x).
\end{equation}
Since both $q_s$ and $w_s^{-1}$ are strictly positive functions, they are in the same differentiability class. 
Moreover, $q_s$ is periodic. A simple calculation shows
\begin{equation}\label{eq:D_s_exact}
\begin{split}
D_s f(x) &= \tfrac{d}{d x}\left[ e^{-\gamma_s x}q_s(x)\,\cdot f(x)\right] \\
&= -\gamma_s e^{-\gamma_s x}q_s(x)\,\cdot f(x)  + e^{-\gamma_s x}q_s'(x)\,\cdot f(x) + e^{-\gamma_s x}q_s(x)\,\cdot f'(x) \\
&= \tfrac{1}{w_s(x)} \,\Big( \Big(-\gamma_s +\tfrac{q_s'(x)}{q_s(x)}\Big)\cdot \mathrm{id} +\tfrac{d}{d x}\Big)f(x)
\end{split}
\end{equation}
and
\begin{align}
D_s\, T_\ell\, f(x) &= \tfrac{d}{d x}\left[ \ e^{-\gamma_s x}q_s(x) f(x-\ell)\right] \\
&= \tfrac{d}{d x}\left[ e^{-\gamma_s \ell}\cdot e^{-\gamma_s (x-\ell)}q_s(x-\ell) f(x-\ell)\right]\\ 
&= e^{-\gamma_s \ell}\ T_\ell\, D_s\, f (x),\qquad \ell\in\Z.
\end{align}
We obtain the commutator rule 
\begin{equation}\label{eq:L_sT_l}
L_s\, T_\ell f = e^{-\ell\,\sum_{k=1}^s\gamma_k}\, T_\ell \,L_s\, f.
\end{equation}
The differential operators $L_s$ are natural in the context of ECC-systems, but in practice, information usually comes as samples of standard derivatives. 
Therefore, in the following statements  we derive the relation between the standard derivatives
and the  differential operators associated with an ECC-system.  

\begin{lemma}\label{lemma:induction_Ls_Dj}
Let $r_s\in C^{m-s}(\R)$, $1\leq s\leq m$, be $1$-periodic, strictly positive functions and $\gamma_1,\dots, \gamma_m\in\R$. Let $L_1,\dots, L_m$ be the differential operators associated with the ECC-system induced by
\begin{equation}
w_s(x) = e^{\gamma_s x}r_s(x),\qquad x\in\R.
\end{equation}
Set for $0\leq s\leq m$
\begin{equation}
v_s(x) = \prod\limits_{k=1}^{s-1} \frac{1}{w_k(x)},\qquad x\in\R.
\end{equation}
For each $0\leq s\leq m$ there exist periodic functions $p_{s,k}\in C^{m-1-s}(\R)$ satisfying
\begin{equation}
L_s f(x) = v_s (x)\cdot \bigg(f^{(s)}(x)+\sum\limits_{k=0}^{s-1} p_{s,k}(x)\cdot f^{(k)}(x)\bigg).
\end{equation}
In a matrix notation, this can be written as
\begin{equation}\label{eq:Ls_matrix_relation}
\big( L_0,\ L_1,\ L_2\ ,\dots,\ L_{s}\big)^t = \V_s \cdot \P_s\cdot 
\Big( \id, \ \tfrac{d}{d x},\ 
\tfrac{d^2}{d x^2},\ 
\dots,\ 
\tfrac{d^s}{d x^s}\Big)^t,
\end{equation}
where $\V_s:\R\to \R^{(s+1)\times (s+1)}$ is
the diagonal matrix  
\begin{equation}
\V_s(x)=\mathrm{diag}(v_0(x),\dots, v_{s}(x))
\end{equation}
and $\P_s:\R\to \R^{(s+1)\times (s+1)}$ is the periodic lower triangular matrix
\begin{equation}
\P_s(x) = \begin{pmatrix}
    1   &   0  & 0  & \cdots  & 0\\[1ex]
p_{1,0}(x) &   1  & 0   & \cdots    & 0\\[1ex]
p_{2,0}(x) & p_{2,1}(x)   & 1  & \cdots  & 0\\[1ex]
\vdots  & \vdots & \vdots  & \ddots   & \vdots\\[1ex]
p_{s,0}(x) & p_{s,1}(x)  & \cdots  & p_{s,s-1}(x) & 1
\end{pmatrix},\qquad x\in\R.
\end{equation}
In particular, $\V_s(x)$ and $\P_s(x)$ are invertible for all $x\in\R$ and are continuous on $\R$.
\end{lemma}

\begin{proof}
We prove this formula by induction. Clearly,  $f^{(0)} = L_0 f$. For the induction step we use \eqref{eq:D_s_exact} and $w_k^{-1}(x) = e^{-\gamma_k x} q_k(x)$. 
\begin{align}
L_{s+1} f &= D_{s+1} L_s f  \overset{\scriptscriptstyle \eqref{eq:D_s_exact}}{=}
\frac{1}{w_{s+1}} \,\bigg( \bigg(-\gamma_{s+1} +\frac{q_{s+1}'}{q_{s+1}}\bigg)\cdot L_s f + (L_s f)'\bigg) \\[.75ex]
 \underset{\scriptscriptstyle \text{step}}{\overset{\scriptscriptstyle \text{Induction}}{\textcolor{white}{=}}}&= \frac{1}{w_{s+1}} \,\bigg\{ \bigg(-\gamma_{s+1} +\frac{q_{s+1}'}{q_{s+1}}\bigg)\cdot v_s \cdot\sum\limits_{k=0}^s p_{s,k}\cdot f^{(k)} \\[.75ex]
& \qquad\quad+ v_s \cdot  \bigg(\,\sum\limits_{k=0}^s p_{s,k}\cdot f^{(k)}\bigg)' + v_s'\cdot\sum\limits_{k=0}^s p_{s,k}\cdot f^{(k)} \bigg)\\[.75ex]
&= \frac{1}{w_{s+1}} \, \bigg(-\gamma_{s+1} +\frac{q_{s+1}'}{q_{s+1}}\bigg)\cdot v_s \cdot\sum\limits_{k=0}^s p_{s,k}\cdot f^{(k)} \bigg\}\\[.75ex]
& \qquad\quad
+ \frac{v_s}{w_{s+1}} \cdot \sum\limits_{k=0}^s p_{s,k}'\cdot f^{(k)} +\frac{v_s}{w_{s+1}} \cdot \sum\limits_{k=0}^s p_{s,k}\cdot f^{(k+1)}\\
&\quad\quad\quad+ \frac{v_s}{w_{s+1}} \cdot\sum\limits_{n=1}^s\bigg(-\gamma_n+\frac{q_n'}{q_n}\bigg)\cdot\sum\limits_{k=0}^s p_{s,k}\cdot f^{(k)},
\end{align}
where we used
\begin{equation}
v_s'= \sum\limits_{n=1}^s {w_n}\cdot v_s\cdot\, \left(w_n^{-1}\right)' = \sum\limits_{n=1}^s {w_n}\cdot v_s\cdot\, w_n^{-1}\bigg(-\gamma_n +\frac{q'_{n}}{q_{n}}\bigg) = v_s\cdot \sum\limits_{n=1}^s \bigg(-\gamma_n +\frac{q'_{n}}{q_{n}}\bigg).
\end{equation}
for the last equality. We recall that $v_{s+1} = \frac{v_s}{w_{s+1}}$ and $p_{s,s} =1$ by the induction hypothesis. We set $p_{s,-1}\coloneqq 0$ to obtain
\begin{align}
L_{s+1} f
&= v_{s+1}
\cdot\bigg( f^{(s+1)}+\sum\limits_{k=0}^s \left(\left(-\gamma_{s+1} +\frac{q_{s+1}'}{q_{s+1}}\right)\cdot p_{s,k} +p_{s,k}' + p_{s,k-1}\right)\cdot f^{(k)} \bigg)\\[.75ex]
&\quad\quad\quad+ v_{s+1} %\frac{v_s}{w_{s+1}}\left( \cdot
\sum\limits_{n=1}^s \left(-\gamma_n+\frac{q_n'}{q_n}\right) \cdot\sum\limits_{k=0}^s p_{s,k}\cdot f^{(k)} \\[.75ex]
&= v_{s+1}\cdot f^{(s+1)}+ v_{s+1} \cdot \sum\limits_{k=0}^s \underbrace{\left(p_{s,k}' + p_{s,k-1}+p_{s,k} \sum\limits_{n=1}^{s+1} \left(-\gamma_{n} +\frac{q_{n}'}{q_{n}}\right)\,\right)}_{=:p_{s+1,k}}\cdot f^{(k)},
\end{align}
concluding the proof.
\end{proof}

\begin{lemma}\label{lemma:M_invertible} 
   Let the knots  $y_i= i-1 , i=1, \dots , D$ and the weights $w_s$, $1\leq s \leq m$ be given and let $B_m^1,\dots, B_m^D=T_{D-1}B_m^1$ be the associated PEB-splines. 
   Then the determinant of the collocation matrix 
 \begin{equation}\label{eq:CB_Matrix_Def}
M \begin{pmatrix}
t_1,\dots, t_D \\
B_m^1,\dots, B_m^D
\end{pmatrix} \coloneqq 
\Big( (B_m^\ell)^{(d_i)}(t_i)\Big)_{1\leq i,\ell\leq D} = \Big( T_{\ell-1}(B_m^1)^{(d_i)}(t_i)\Big)_{1\leq i,\ell\leq D}
\end{equation}
is nonnegative. It is positive if and only if 
\begin{equation}
t_i\in \begin{cases}
(i-1,i+m-1), & d_i<m-1 \\
[i-1,i+m-1), & d_i = m-1,
\end{cases}
\qquad 1\leq i\leq D.
\end{equation}
\end{lemma}
\begin{proof}
We first address the second equality in the definition of the matrix $M$. As already noted in \eqref{eq:ShiftInvariance_PEB}, the PEB-splines  for equispaced knots are shifts of a single PEB-spline, namely, $B_m^\ell = T_{\ell - 1} B_m^1$. Since differentiation and translation operators commute, we have $(B_m^\ell)^{(d_i)}(t_i) = T_{\ell-1}(B_m^1)^{(d_i)}(t_i)$. 

Let as before $d_i$ count the previous occurrences of the point $t_i$, i.e.,
\begin{equation}
d_i=\max\left\lbrace \ell : t_i =\dots = t_{i-\ell}\right\rbrace\leq m-1,\qquad 1\leq i\leq D.
\end{equation} 
Then $t = (t_1,t_2, t_3, \dots, t_D)$ can be rewritten as 
$$
  t = \Big( \underbrace{t_1,\dots, t_1}_{{d_1+1} \text{-times}},
  \underbrace{t_{d_1+2},\dots, t_{d_1+2}}_{d_2+1  \text{-times}}, \underbrace{t_{d_1+d_2+3},\dots, t_{d_1+d_2+3}}_{d_3+1  \text{-times}} , \dots, \underbrace{t_D,\dots, t_D}_{{d_N+1} \text{-times}}\Big).
$$
We define the block-diagonal matrices 
\begin{equation}
\PP(t):= \begin{pmatrix}
    \P_{d_1}(t_1)   &   0  & 0  & \cdots  & 0\\[1ex]
0 &   \P_{d_2}(t_{d_1+2})  & 0  & \cdots  & 0\\[1ex]
0 & 0   & \P_{d_3}(t_{d_1+d_2+3})  & \cdots  & 0\\[1ex]
\vdots  & \vdots  & \vdots  & \ddots  & \vdots\\[1ex]
0 & 0   & \cdots  & 0 & \P_{d_N}(t_D)
  \end{pmatrix}\in \R^{D\times D},
\end{equation}
\begin{equation}
\VV(t):= \begin{pmatrix}
    \V_{d_1}(t_1)   &   0  & 0  & \cdots  & 0\\[1ex]
0 &   \V_{d_2}(t_{d_1+2})  & 0  & \cdots  & 0\\[1ex]
0 & 0  & \V_{d_3}(t_{d_1+d_2+3})  & \cdots & 0\\[1ex]
\vdots  & \vdots  & \vdots  & \ddots  & \vdots\\[1ex]
0 & 0   & \cdots  & 0 & \V_{d_N}(t_D)
  \end{pmatrix} \in \R^{D\times D}.
\end{equation}
Since each $\P_{d_i}$ has determinant $1$,  the block  matrix $\PP(t)$ also has determinant $1$. By construction, the matrix $\VV(t)$ is a diagonal matrix with strictly positive diagonal entries. 
By Lemma \ref{lemma:induction_Ls_Dj}, the collocation matrix with the associated differential operators and the collocation matrix with the standard derivatives are related by 
\begin{equation}
    \left( L_{d_i}B_m^\ell(t_i)\right)_{1\leq i,\ell\leq D} = \VV(t)\cdot\PP(t)\cdot M \begin{pmatrix}
t_1,\dots, t_D \\
B_m^1,\dots, B_m^D
\end{pmatrix}.
\end{equation}
The claim follows from the multiplicativity of the determinant and Theorem \ref{thm:SW-cond_OV} with $y_i=i-1$.
\end{proof}

We finally relate the collocation matrix to sampling with multiplicities.

\begin{corollary}\label{cor:matrix_multiplicities}
Let $\varphi$ be a PEB-spline of order $m$ supported on the interval $[0,m]$. Further let $x= \left(x_1,x_2,\dots,x_N\right)\in\R^N$ be a vector of pairwise distinct sampling points, $x_1 <\dots < x_N$, with multiplicities $\mu = \left(\mu_1,\dots, \mu_N\right)\in\{0,\dots, m-1\}^N$. Set $D = \sum_{j=1}^N (1+\mu_j)$. The Hermite interpolation problem
\begin{equation}\label{eq:Hermite_to_translate}
    f=\sum\limits_{\ell=\ell_0}^{\ell_0+D-1} c_\ell T_\ell \varphi,\quad f^{(s)}(x_j) = \y_{j,s},\quad \ 0\leq s\leq \mu_j,\ 1\leq j\leq N
\end{equation}
has a unique solution if and only if 
\begin{equation}\label{eq:x_pos_condition}     \begin{split}
    x_j &\in (\mu_j+\ell_0, m+ \ell_0) +\sum\limits_{n=1}^{j-1} \left(1+\mu_n\right),\qquad \text{if}\ \ \mu_j <m-1, \\[.5ex]
    x_j& \in [\mu_j+\ell_0, m+ \ell_0) +\sum\limits_{n=1}^{j-1} \left(1+\mu_n\right),\qquad \text{if}\ \ \mu_j = m-1. 
\end{split}
\end{equation}

In this case, the solution depends linearly on the interpolation data.
\end{corollary}

If the points satisfy \eqref{eq:x_pos_condition}, we say that they \emph{satisfy the Schoenberg-Whitney conditions}.
\begin{proof}
To apply the interlacing property to the interpolation problem \eqref{eq:Hermite_to_translate}, we set 
\begin{equation}
t = \Big( \underbrace{x_1,\dots, x_1}_{{\mu_1+1} \text{-times}}, \underbrace{x_2,\dots, x_2}_{{\mu_2+1} \text{-times}},\dots, \underbrace{x_N,\dots, x_N}_{{\mu_N+1} \text{-times}}\Big)\in \R^D, 
\end{equation}
to be precise,
\begin{equation}
t_i = x_j \qquad \text{ if } \qquad \sum\limits_{n = 1}^{j-1} (1+\mu_n) < i\leq \sum\limits_{n= 1}^{j} (1+\mu_n).
\end{equation} 
Then the vector of multiplicities is given by
\begin{equation}
d = \Big( 0, 1, \dots, \mu_1, 0, 1,\dots, \mu_2, \dots, 0,1,\dots,\mu_N\Big)\in \big\lbrace 0,\dots, m-1\big\rbrace^D.
\end{equation}
Accordingly, we set $\y\in\C^D$ to be the list of the samples $\y = (\y _i)_{i=1, \dots , D}$, 
\begin{equation}
\y = \Big( \y_{1,0},\, \y_{1,1},\, \y_{1,2},\, \dots,\,
\y_{1,\mu_1},\, \y_{2,0},\, \y_{2,1},\,\y_{2,2},\,\dots,\,
\y_{2,\mu_2},\, \dots,\, \y_{N,0},\,\y_{N,1},\,
\y_{N,2}\dots,\,\y_{N,\mu_n}\Big).
\end{equation}
With this notation, the interpolation problem \eqref{eq:Hermite_to_translate} can be recast as
\begin{equation}
f=\sum\limits_{\ell=\ell_0}^{\ell_0+D-1} c_\ell\, T_\ell \varphi,\quad f^{(d_i)}(t_i) = \y_i,\qquad 1\leq i\leq D.
\end{equation}
The corresponding collocation matrix is 
\begin{equation} 
M \begin{pmatrix}
t_1,\dots, t_D \\
T_{\ell_0}\varphi,\dots, T_{\ell_0+D-1}\varphi
\end{pmatrix} = \Big( T_{\ell-1}\varphi^{(d_i)}(t_i-\ell_0)\Big)_{1\leq i,\ell\leq D}.
\end{equation}
For further reference, we denote it by 
\begin{equation}\label{eq:M_corollary_def}
\MM(x,\mu,\ell_0) \coloneqq M 
\begin{pmatrix}
t_1,\dots, t_D \\
T_{\ell_0}\varphi,\dots, T_{\ell_0+D-1}\varphi
\end{pmatrix} . 
\end{equation}
By Lemma \ref{lemma:M_invertible}, the collocation matrix is
invertible, and thus a unique solution of \eqref{eq:Hermite_to_translate} exists, if and only if $t_i -\ell_0\in (i-1,i+m-1)$ for all $1\leq i\leq D$ (with the modification $t_i-\ell_0\in [i-1,i+m-1)$ if $d_i=m-1$). Expressed with the sampling points $x_j$, the collocation matrix is invertible if and only if
\begin{align}
x_j \ &\in \ \bigcap_{s = 0}^{\mu_j} \
\Big(\left(s, s+m\right)\ +\ell_0+\sum\limits_{n=1}^{j-1} \left(1+\mu_n\right)\Big)
= (\mu_j+\ell_0, m+\ell_0)+\sum\limits_{n=1}^{j-1} \left(1+\mu_n\right). 
\end{align}
for all $1\leq j \leq N$. The second condition in \eqref{eq:x_pos_condition} is the adaptation for $\mu_j=m-1$.  
\end{proof}

\subsection{Stability of PEB-Splines}\label{subsec:stability_PEB} 
To use PEB-splines as generators of \sis s, we need to verify the stability of PEB-splines and their derivatives.

\begin{lemma}\label{lemma:PEB_stable_shifts}
Let $\varphi$ be a PEB-spline of order $m$. Then 
\begin{enumerate}[(i)]
\item The generator $\varphi$ has stable integer shifts.
\item For all $0\leq s\leq m-1$, the tuple $\left(\varphi, \dots, \varphi^{(s)}\right)$ has stable integer shifts.
\end{enumerate}
\end{lemma}
\begin{proof} 
(i) Without loss of generality, we can assume that  $\varphi$ is supported on $[0,m]$. 
 By Theorem \ref{thm:stable_shifts}, it  suffices to show the
$\ell^\infty$-independence of $(T_\ell \varphi)_{\ell\in\Z}$. To that
end, let  $c\in\ell^\infty(\Z), c \neq 0$, and assume that 
$c_{\ell_0}\neq 0 $. Since 
$\supp(T_\ell \varphi) = [\ell, \ell+m]$, the restriction of $f$ to the  interval $(\ell_0+m-1, \ell_0+m)$ is given by 
\begin{equation}
f = \sum\limits_{\ell=\ell_0}^{\ell_0+m-1}c_\ell T_\ell \varphi.
\end{equation}
The restriction of $f$ to $(\ell_0+m-1,\ell_0+m)$ clearly belongs to the ECC-system. By Lemma \ref{thm:Tzeros}, it has at most $m-1$ zeros. In particular, it does not vanish on this interval.

(ii) follows from  \textrm{(i)}, because the inequality  
\begin{equation} 
    A_p \norm{c}_{\ell^p}^p \leq \Big\lVert\,\sum_{\ell\in\Z}c_\ell T_\ell\varphi\,\Big\rVert_{L^p}^p \leq \sum\limits_{s=0}^{S} \Big\lVert\sum_{\ell\in\Z}c_\ell T_\ell\varphi^{(s)}\,\Big\rVert_{L^p}^p.
  \end{equation}
  implies that the generator $\Phi = (\varphi , \varphi ' , \dots
  \varphi ^S)$ 
  has stable integer shifts. The upper inequality in~\eqref{c1} are
  obvious. 
\end{proof}

\subsection{EB-splines} Our  initial motivation to consider PEB-splines was the  particular
subclass of  exponential B-splines. These  have already been used in time-frequency analysis
in \cite{BannertEtAl2014, KloosStoeckler2014}. Whereas general
PEB-splines are usually defined recursively via divided differences,  EB-splines also have a direct  definition by their  Fourier transform as follows from \cite{ChristensenMassopust2011}. Let $m\in\N$ and let $\alpha =\left( \alpha_1, \dots, \alpha_m \right)\in\R^m$. 
An EB-spline $\mathcal{E}_{m,\alpha}: \R\longrightarrow\R$ of order $m$ for the $m$-tuple $\alpha$ is a function of the form 
\begin{equation}
\mathcal{E}_{m,\alpha}(x):={\prod\limits_{s=1}^{m}}^{\mbox{\raisebox{-7pt}{\large $ *$}}}\, e^{\alpha_s x}\indicator_{[0,1)}(x),
\end{equation}
where $\prod^{\mbox{{\large $ *$}}}$ denotes the convolution product. Its Fourier transform is given by
\begin{equation}\label{c2}
\widehat{\mathcal{E}}_{m,\alpha}(\omega) = \prod\limits_{s=1}^m\frac{e^{\alpha_s - 2\pi i \omega}-1}{\alpha_s - 2\pi i \omega}.
\end{equation} 
If $\alpha_1 = \dots = \alpha_m = 0$, then this is a classical B-spline.
To establish the connection with the original definition of PEB-splines, we set $\gamma_1:=\alpha_1$ and $\gamma_s = \alpha_s-\alpha_{s-1}$ for $2\leq s\leq m$. Then the weight functions $w_s(x) = e^{\gamma_s x}$ induce the EB-spline $B_m = \mathcal{E}_{m,\alpha}$.
For an  EB-spline, the stability follows directly from the factorization~\eqref{c2} and Theorem~\ref{thm:stable_shifts}(iv). 

\section{A Sampling Theorem via Uniform Bounds on the Collocation Matrices }\label{sec:main1} 
In this section, we prove a first sampling theorem with derivatives in a \sis\ generated by a PEB spline. The proof technique is inspired by \cite{AldroubiGroechenig2000}, and its main point is to establish uniform bounds on a family of collocation matrices. The use of derivatives  requires more spline theory and leads to more technicalities.

Our  standing assumption is that the generator $\varphi$ is a periodic exponential B-spline of order $m$,  as described in Subsection \ref{subs:PEB-splines}. Without loss of generality, we  assume that $\varphi$ is supported on $[0,m]$. We always assume that $X$ is
separated, i.e.,  $\inf _{x,y\in X,x\neq y} |x-y| = \delta >0$, and $\delta $ is called the separation constant of $X$. 

As in \cite{AldroubiGroechenig2000}, we partition $\R $ into suitable intervals and analyze the sampling  problem locally. Given integers $M\in\Z$ and $L\in\N$, we partition $\R$ in intervals
\begin{equation}
\R = \bigcup\limits_{k\in\Z}[M+kL, M+kL+L) = \bigcup\limits_{k\in\Z}I_{M,L}(k).
\end{equation}

\begin{remark}
Let $m\in \N$ and let $\ell_0\leq x_1 < \dots <x_N \leq \ell_0 +L$ be sampling locations with multiplicities $\d_1,\dots,\d_N \in \{0,1\dots, m-1\}$ such that 
$$
    \sum\limits_{j=1}^{N}(1+\d_j) = L+m-1
$$
and $\d_N\leq m-2$ if $x_N= \ell_0+L$. Then for all $(c_\ell)_{\ell\in\Z}\in\ell^2(\Z)$ we have 
\begin{equation}\label{eq:rem_restriction}
    \Big(\sum\limits_{\ell\in\Z} c_\ell T_\ell \varphi\Big)^{(s_j)}(x_j) = \Big(\sum\limits_{\ell=\ell_0- m+1}^{\ell_0+ L-1} c_\ell T_\ell \varphi\Big)^{(s_j)}(x_j), \qquad 0\leq s_j\leq \d_j,\ 1\leq j\leq N.
\end{equation}
Since $\varphi$ vanishes outside of $[0,m]$, \eqref{eq:rem_restriction} is clear for all $x_1,\dots ,x_{N-1}$.  
The additional assumption that $\d_N\leq m-2$ if $x_N = \ell_0+L$ implies that the measurements at the last sampling point do not include a point of discontinuity of $T_{\ell_0+L}\varphi ^{(m-1)}$.
Thus for all $0\leq s\leq \d_{N}$ and all $\ell \geq \ell_0 +L$, $T_{\ell}\varphi^{(s)}(\ell _0 + L)=0$. These observations imply \eqref{eq:rem_restriction}. 
    
\end{remark}
To avoid unnecessary banalities, we first settle  the easy case of PEB-splines of order $1$.  
 
\begin{lemma}
Let $\varphi$ be a PEB-spline of order 1 supported on $[0,1)$. Assume that $X\subseteq \R$ is separated. Then the following are equivalent:
\begin{enumerate}[(i)]
\item For every $k\in\Z$, there exists a point $x^k\in X\cap [k,k+1)$.
\item $(X, \mu)$ is a sampling set for $V^p(\varphi)$ for all $p\in[1,\infty]$.
\item $(X, \mu)$ is a uniqueness set for $V^p(\varphi)$ for all $p\in[1,\infty]$.
\end{enumerate}
\end{lemma}
\begin{proof}
The implication  $(ii) \Rightarrow (iii)$ is trivial. A PEB-spline of the first order is the cut-off of the continuous weight function, namely $\varphi = w_1\cdot\indicator_{[0,1)}$. Thus, if there exists a $k\in\Z$ with $X\cap [k,k+1)= \emptyset$, then $T_k\varphi|_{X} =0$, although this function is clearly not the zero function. This shows $(iii)\Rightarrow(i)$.

Now assume that $(i)$ holds.
The upper sampling bound is satisfied by Theorem  \ref{thm:upper_bound_W_jumps}. Since $w$ does not vanish, $\varphi$ has no zeros on $[0,1)$ and vanishes outside this interval. Therefore, $\varphi(x^k-\ell) \neq 0$ holds if and only if $k=\ell$.
Since $w$ is continuous and is strictly positive on $\R$, it attains a minimum $a>0$ on the compact interval $[0,1]$. For all $p\in[1,\infty)$ we obtain
\begin{align}
\sum\limits_{x\in X}\abs{f(x)}^p & \geq\sum\limits_{k\in\Z} \abs{f(x^k)}^p  = \sum\limits_{k\in\Z} \abs{c_{k}\varphi(x^k - k)}^p \geq a^p \sum\limits_{k\in\Z} \abs{c_{k}}^p ,
\end{align}
with an analogous estimate for $p=\infty$. This shows $(i)\Rightarrow(ii)$ and completes the proof.
\end{proof}
From now on, we assume without loss of
generality that $m\geq 2$. Our first theorem looks a bit technical, but is definitely useful. 

\begin{theorem}\label{thm:main_compact} 
Let $\varphi$ be a PEB-spline of order $m\in\N$ supported on $[0,m]$. Assume $X\subseteq \R$ with multiplicity function $\mu_X$ satisfies the following properties:
\begin{enumerate}[(i)]
\item $X$ is separated. % for some $\delta>0$.
\item $S:=\max\limits_{x\in X} \mu_X(x)\leq m-1$.
\item There exist integers $M, L>0$
and $\varepsilon >0$, such that for every $k\in\Z$, there exist points $x_1^k < x_2^k<\dots <x^k_{N(k)}$ in $X\cap I_{M,L}(k)$ and nonnegative integers $\d_1^k ,\, \d_2^k,\, \dots,\, \d^k_{N(k)}$ with the following three properties:
\begin{equation}\label{eq:thm_compact_cond3.1} 
\d^k_j \leq \mu_X(x^k_j)\qquad\text{ for all }k\in\Z,\, 1\leq j\leq N(k),
\end{equation}
\begin{equation}\label{eq:thm_compact_cond3.2}
\sum\limits_{j=1}^{N(k)} \left(1+\d_j^k\right) = L+m-1 \qquad\text{ for all }k\in\Z,
\end{equation}
\begin{equation}\label{eq:thm_compact_cond3.3}
x_j^k 
 \in X\cap \Big(\left[ \d^k_j-m+\varepsilon, -\varepsilon\right] +M+kL+1+\sum\limits_{n=1}^{j-1} \left(1+\d^k_n\right)\Big).
\end{equation}
\end{enumerate}
Then $(X, \mu_X)$ is a sampling set for $V^p(\varphi)$ for all $p\in[1,\infty]$.
\end{theorem}
\begin{remark}
The conditions \eqref{eq:thm_compact_cond3.1} - \eqref{eq:thm_compact_cond3.3}
are compatible with \eqref{eq:rem_restriction} and in some sense are  necessary. Since the restriction of $f$ to an interval $I_{M,L}(k)$ belongs to an $(L+m-1)$-dimensional space, we need at least $L+m-1$ samples, whence \eqref{eq:thm_compact_cond3.2}. If more information is available, we either remove points or the highest derivatives, which is \textit{(ii)} and \eqref{eq:thm_compact_cond3.2}. As we intend  to use the invertibility of the collocation matrix (Corollary \ref{cor:matrix_multiplicities}), the retained information must be chosen diligently and involves a Schoenberg-Whitney condition of type \eqref{eq:thm_compact_cond3.3}, which says that the support of each shift must contain a sampling point (see Figure \ref{fig:necessary_points}). 

\begin{figure}[h!tp]
		\includegraphics[width=\textwidth]{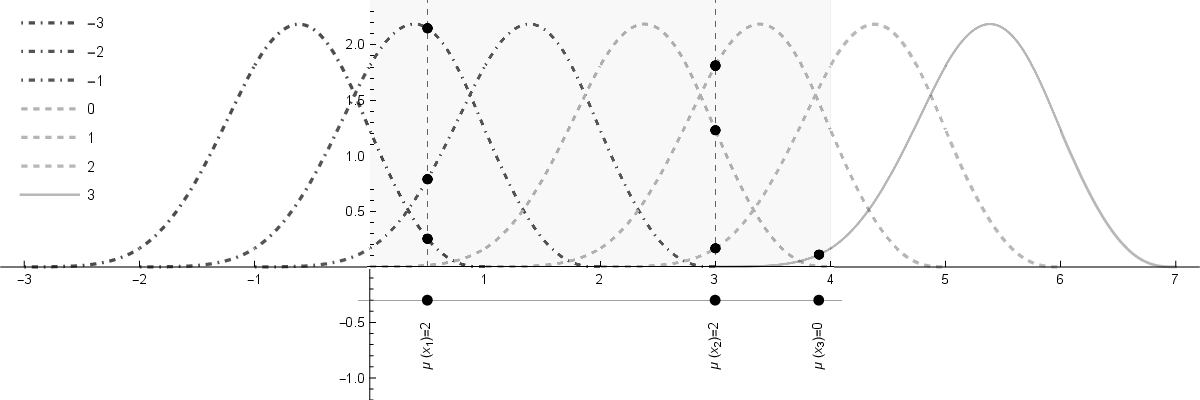}
	\caption{Nonvanishing shifts of $\varphi(x) = {\prod_{j=1}^{\mbox{\large $*$}\ 4\ }}\, e^{x}\indicator_{[0,1)}(x)$. The sampling points are $x_1 = 0.5$, $x_2=3$, $x_3=3.9$, with multiplicities $\mu_{X}(x_1)=\mu_{X}(x_2)=2$, $\mu_{X}(x_3)=0$. The first sampling point lies in the support of the first three shifts of $\varphi$ (dot-dashed), the second point is in the support of the next three shifts of $\varphi$ (dashed), and the last point - in the support of the last shift of $\varphi$ (solid).}
	\label{fig:necessary_points}
\end{figure}
\end{remark}
\begin{proof}
Fix $p\in[1,\infty)$ and let
\begin{equation}
f = \sum_{\ell\in\Z} c_\ell T_\ell \varphi\in V^p(\varphi).
\end{equation} 
Since $\norm{f}_p\asymp \norm{c}_p$ for $f\in V^p(\varphi)$ by Lemma \ref{lemma:PEB_stable_shifts}, it suffices to show the inequality 
\begin{equation} 
\sum\limits_{x\in X}\sum\limits_{s=0}^{\mu_X(x)} \abs{f^{(s)}(x)}^p\asymp \norm{c}_{\ell^p}^p.
\end{equation}
For given $M,\,L\in \N$ every interval $I_{M, L}(k)=[M+kL,M+(k+1)L)$ will be investigated separately. The aim is to use Corollary \ref{cor:matrix_multiplicities}. 

\textbf{Step 1: The upper bound.} 
Since $\varphi^{(s)}$, $0\leq s\leq m-1$, are piecewise continuous, compactly supported functions with finitely many jump discontinuities, Theorem \ref{thm:upper_bound_W_jumps} (ii) applies and yields the upper bound for the sampling inequality. 

\textbf{Step 2: A class of collocation matrices with uniform bounds.} 
Let $\mathcal{M}$ be the set
\begin{equation}\label{eq:compact_posible_mu}
\mathcal{M}\coloneqq \Big\lbrace \d\in\lbrace 0,\dots , S\rbrace^N : N\in\N,\ \sum\limits_{j=1}^N \left( 1+\d_j\right) = L+m-1\Big\rbrace. 
\end{equation}
If $\d\in\mathcal{M} \cap \Z^N$, then $N\leq L+m-1$. Therefore, 
\begin{equation}
\mathcal{M}\subseteq \bigcup\limits_{N=1}^{L+m-1}\left\lbrace 0,\dots,
  S \right\rbrace^{N} \, ,
\end{equation}
in particular, the set $\mathcal{M}$ is finite. Let $\Omega_\d\subseteq \R^{N}$ for $\d\in \mathcal{M} \cap\Z^N$ be the set of vectors $z = ( z_1, z_2,\dots,z_N)^t\in\R^{N}$ with 
\begin{equation}
\min\limits_{i\neq j} \abs{z_i-z_j}\geq \delta
\end{equation}
and 
\begin{equation}\label{eq:thm_compact_z}
    \begin{split}
z_j&\in [0, L]\cap \
\Big(\left[\d_j-m+\varepsilon, -\varepsilon\right]\ +1+\sum\limits_{n=1}^{j-1} \left(1+\d_n\right)\Big) \\
&= \Big[ \max\Big\lbrace 0, \d_j-m+1+\varepsilon+ \sum\limits_{n=1}^{j-1} \left(1+\d_n\right)\Big\rbrace\, ,\, \min\Big\lbrace L, 1-\varepsilon +\sum\limits_{n=1}^{j-1} \left(1+\d_n\right)\Big\rbrace\Big].
    \end{split}
\end{equation}
In other words, $\Omega_\d$ contains all point configurations in $[0,L]$ that satisfy  Condition \eqref{eq:thm_compact_cond3.3}. The set is bounded and closed, hence compact in $\R^{N}$. For each $z\in\Omega_\d$ we consider the collocation matrix
\begin{equation}
\MM(z, \d, -m+1)\in\R^{(L+m-1)\times (L+m-1)}
\end{equation} 
as defined in the proof of Corollary \ref{cor:matrix_multiplicities}
in \eqref{eq:M_corollary_def} with parameters $D = L+m-1$ and $\ell_0 = -m+1$. 
By Corollary \ref{cor:matrix_multiplicities} and the positions of the points (cf. \eqref{eq:thm_compact_z}), the matrix is invertible. In addition, $z\mapsto \MM(z, \d, -m+1)$ is continuous on the compact set $\Omega_\d$, so the matrices are bounded from above and below uniformly in $\omega\in\Omega_\d$. 
To see this, a trivial upper bound is a multiple of the maximum of the entries of $\MM(z, \d, -m+1)$. For the lower bound, we note that $z\mapsto \MM(z, \d, -m+1)^{-1}$ is continuous by Cramer's rule, thus we can again obtain a uniform upper bound $0<a_p(\d)^{-1}$ for $\MM(z,
\d, -m+1)^{-1}$. Then $a_p(\d)$ is a strictly positive uniform lower bound of $\MM(z, \d, -m+1)$, i.e., for all $c\in\R^{L+m-1}$ 
\begin{equation}
\norm{\MM(z, \d, -m+1)\cdot c}_p \geq a_p(\d)\norm{c}_p. 
\end{equation}
Since $\mathcal{M}$ is finite, we can set 
\begin{equation}
A_p\coloneqq \min\lbrace a_p(\d) : \d\in\mathcal{M}\rbrace>0.
\end{equation}
\textbf{Step 3: The lower estimate.}
Let $x^k = (x_1^k, \dots, x_{N(k)}^k)$ be the points in $X \cap I_{M,L}(k)$ that satisfy Condition \textit{(iii)}, i.e., they  satisfy \eqref{eq:thm_compact_cond3.1}, \eqref{eq:thm_compact_cond3.2}, and \eqref{eq:thm_compact_cond3.3}. We set $\d^k = \left(\d_1^k,\dots, \d^k_{N(k)}\right)\in\mathcal{M}$ and $z_j^k = x_j^k-(M+kL)$. Then $z^k\in \Omega_{\d^k}$. 
Let $\textbf{F}(x^k)$  be the complete list of evaluations 
\begin{align}\label{eq:thm_compact:list_of_Eval}
    \textbf{F}(x^k) &:= \left( f(x^k_1),\dots, f^{(\d^k_1)}(x^k_1),
    \dots,  f(x_{N(k)}^k),\dots, f^{(\d^k_{N(k)})}(x_{N(k)}^k)\right) \in \C^{L+m-1}.
\end{align}
For $p\in[1,\infty)$, we estimate from below the sum
\begin{equation}
    \sum\limits_{x\in X}\sum\limits_{s=0}^{\mu_X(x)}\abs{f^{(s)}(x)}^p \geq \sum\limits_{k\in\Z}\sum\limits_{j=1}^{N(k)}\sum\limits_{s=0}^{\d_j^k}\abs{f^{(s)}(x_j^k)}^p = \sum\limits_{k\in\Z} \norm{\textbf{F}(x^k)}_{p}^p.
\end{equation}
We now consider the restriction of $f$ to the interval $I_{M,L}(k)$ and shift it  to $[0,L)$. The resulting function is  
\begin{equation}
    T_{-(M+kL)} \big( f \big| _{I_{M,L}(k)} \big) =  T_{-(M+kL)}\,\sum\limits_{\ell = M+kL-m+1}^{M+kL+L-1} c_\ell T_\ell \varphi = \sum\limits_{\ell =-m+1}^{L-1} c_{\ell+M+kL} T_{\ell} \varphi.
\end{equation}
We denote by $\mathfrak{c}_k = \left(c_{M+kL-m+1},\dots, c_{M+kL+L-1}\right)^t$ the extracted coefficients which are active on $I_{M,L}(k)$. 
By Corollary \ref{cor:matrix_multiplicities}, specifically \eqref{eq:M_corollary_def}, the measurements $\textbf{F}(x^k)$ are given by
\begin{align}
\textbf{F}(x^k) = 
\MM(z^k,\d^k,-m+1)\cdot \mathfrak{c}_k.
\end{align}
From this follows
\begin{equation}
\begin{split}
\sum\limits_{x\in X}\sum\limits_{s=0}^{\mu_X(x)}\abs{f^{(s)}(x)}^p & \geq  \sum\limits_{k\in\Z} \norm{
\MM(z^k,\d^k,-m+1)\cdot \mathfrak{c}_k }_p^p 
\geq A_p^p \, \sum\limits_{k\in\Z} \norm{\mathfrak{c}_k}_p^p \geq
A_p^p \,  \norm{c}_p^p\, .
\end{split}
\end{equation}
In the last inequality, we have used the fact that every $\ell \in \Z $ is
contained in some interval $   M+kL-m+1 \leq \ell\leq M+kL+L-1.$ 
The estimate for $p=\infty$ is analogous, and we are done. 
\end{proof}

\begin{remark}
The assumptions in Theorem \ref{thm:main_compact} are not easy to
visualize. We will therefore  investigate a different set of conditions in Section \ref{sec:main2}.
But Theorem~\ref{thm:main_compact} is a powerful result, and the
imposed conditions are necessary and optimal in several regards.

(i) The restriction of $f\in V^p(\varphi )$ to the   interval $I_{M,
  L}(k)$ is given by~\eqref{eq:rem_restriction} and is a linear combination
of $L+m-1$ linearly independent shifts $T_{\ell}\varphi $, this is
Condition \eqref{eq:thm_compact_cond3.2}).

(ii) To recover the restriction of $f$ on the interval $I_{M,L}(k)$ from $L+m-1$
samples, the associated collocation matrix must be invertible. This is implied by
Condition~\eqref{eq:thm_compact_cond3.3}. 
The margin $\varepsilon >0$ leads to uniform estimates in $k$. 

(iii)  Condition \eqref{eq:thm_compact_cond3.2} implies that $(X,\mu_X)$  contains a subset $X'\subseteq X$ with multiplicity function $\mu _{X'}$  satisfying  $\mu _X'(x) \leq \mu _X(x),$ $x\in X'$, such that $D^-(X',\mu _{X'} )= 1+\frac{m-1}{L}$, cf. \cite[Prop.~3.7.]{GroechenigEtAl2019}. Thus every sampling set subject to the assumptions of Theorem~\ref{thm:main_compact} contains a sampling set with a weighted density close to the  necessary density condition.
\end{remark}

To illustrate the strength of Theorem~\ref{thm:main_compact}, we draw two consequences. First, we check standard sampling sets without multiplicities. It will be convenient to label $X$ as a bi-infinite strictly increasing sequence $X=(x_j)_{j\in \Z}$, $ x_j < x_{j+1}$. In
this notation the set $X$ is $\delta $-separated if $\inf _{j\in \Z} (x_{j+1}-x_j)\geq \delta >0$, and its maximum gap is defined as
\begin{equation}\label{eq:max_gap_OV}
\mg(X)\coloneqq\sup_{j\in \Z}\, (x_{j+1}-x_j) \, .
\end{equation}

The following result extends \cite[Thm.~4]{AldroubiGroechenig2000}
 from B-splines to general  PEB-splines.
\begin{corollary}\label{cor:max_gap_OV} 
Let $\varphi$ be a PEB-spline of order $m\in\N$.
If the sampling set $X\subseteq\R$ is separated and the maximum gap  satisfies $\mg(X)<1$, then $X$ is a sampling set for $V^p(\varphi)$.
\end{corollary}
If $\mu_X\equiv 0$, then the assumptions in Theorem
\ref{thm:main_compact} are identical with those in
\cite[Thm.~3]{AldroubiGroechenig2000}. Further, the proof of \cite[Thm.~4]{AldroubiGroechenig2000} shows that a set with $\mg (X) <1$  satisfies the conditions in \cite[Thm.~4]{AldroubiGroechenig2000}, hence, the conditions in Theorem \ref{thm:main_compact}. We will give an alternative proof in Section~\ref{sec:main2}. 

To construct sampling sets with multiplicities, one can start with a sampling set without derivatives and remove certain sampling points while adding multiplicities at neighbouring sampling points (see Figure \ref{fig:preserving_rearrangement}).

Finally, we show that there exist sets with a given
multiplicity function that satisfies the sufficient conditions of Theorem~\ref{thm:main_compact} with a density arbitrarily close to the necessary condition.
\begin{figure}[h!tp]
	\subfigure[Starting arrangement.]{
		\includegraphics[width=.47\textwidth]{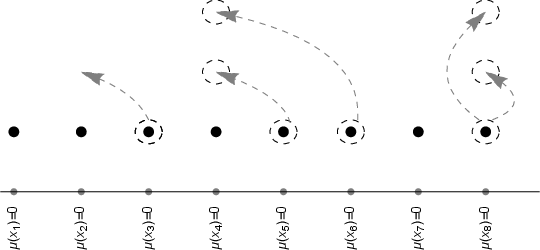}
	}
	\hfill
	\subfigure[Modification with multiplicities.]{
		\includegraphics[width=.47\textwidth]{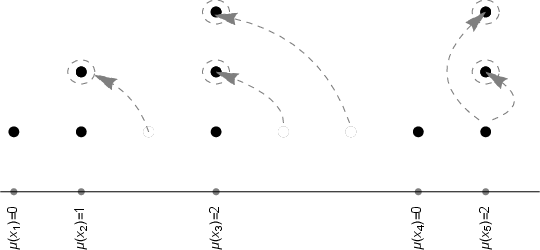}
	}
	\caption{Local modification of a sampling set that preserves the conditions of Thm. \ref{thm:main_compact}.}
	\label{fig:preserving_rearrangement}
\end{figure} 

 \begin{theorem}\label{thm:existence}
   For every $\nu >0$  and given multiplicity sequence  $\mu =
   (\mu _j)_{j\in \Z}$ with values in $\{0,\dots m-1\}$, there
   exists a  sequence  $X= (x_j)_{j\in \Z}$, such that 
   $(X,\mu_X)$ is sampling for $V^p(\varphi)$, $\mu_X(x_j)=\mu_j$ for all $j\in\Z$, and $D^-(X,\mu_X)<
   1+\nu $.
 \end{theorem}

 \begin{proof}
   The sampling set $X$ will be a suitable subset of a sufficiently dense lattice $\alpha
   \Z $. In line with Theorem~\ref{thm:main_compact} we use a suitable
   partition and argue locally. 
   Since $X$ will be a subset of $\alpha \Z$, it is separated, and
   Condition (ii) in Theorem \ref{thm:main_compact} is satisfied by
   assumption. The main part is to construct  points $x$ and submultiplicities $\tilde \d$ to satisfy the third condition.

\textbf{Step 1. Global parameters.}
Given $\nu >0$, choose $L \in \N $ such that $\frac{2m-2}{L}<\nu$. We set $\varepsilon = \alpha = \frac{1}{(m+1)}$. We will work with the interval partitioning $ I_{L,0}(k) = [kL,(k+1)L)$. The key observation here is the fact that $\#( [\ell+\varepsilon, \ell+1-\varepsilon] \cap \alpha \Z) =m$, and for a fixed $\ell\in\Z$ 
\begin{equation}\label{eq:ce1}
\begin{split}
    [\ell+\varepsilon, \ell+1-\varepsilon]
    \cap \Big(\left[\d^k_j-m+\varepsilon, -\varepsilon\right] +kL+1+\sum\limits_{n=1}^{j-1} \left(1+\d^k_n\right)\Big) 
\end{split}
\end{equation}
is non-empty for at most $m$ consecutive indices $j, j+1,\dots, j+m-1$. Thus, in 
$[\ell+\varepsilon, \ell+1-\varepsilon]$ we will need at most $m$ sampling points in $\alpha \Z$, which are available by choice of $\alpha$.

\textbf{Step 2. Local construction.} Let $N\in \N$ be the minimal natural number with 
$
\sum_{j=1}^N (1+\mu_j) \geq L+m-1.
$
We set  
\vskip-13pt
$$ 
    \tilde \d_j \coloneqq \mu_j, \quad 1\leq j\leq N-1, \qquad \tilde \d_N \coloneqq (L+m-1) -
    \sum\limits_{j=1}^N (1+\mu_j) -1.
$$
   
The minimality of $N$ implies that $0\leq \tilde \d_N\leq \mu_N$, and the definition of $\tilde \d_N$ implies that 
\begin{equation*}
    \sum\limits_{j=1}^{N} (1+\tilde\d_j) = L+m-1.
\end{equation*}
   
\begin{figure}[h!tp] 
    \includegraphics[width=\textwidth]{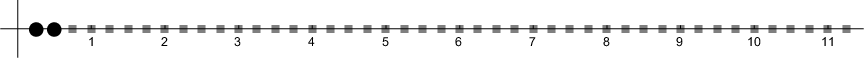}
    \includegraphics[width=\textwidth]{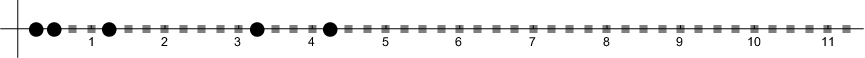}
    \includegraphics[width=\textwidth]{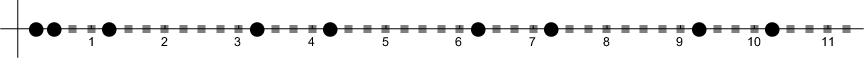}
	\caption{Choice of points (black circles) for $m=3$, $\mu = (1+(-1)^j)_{j\in\Z}$ and $\nu=\frac{1}{2}$ from the lattice $\frac{1}{4}\Z$ (gray squares). The interval length is $L=13$, and the number of points is $N=9$. The sampling locations are $\alpha k,\, k\in\{1, 2,5,13,17,25,29,37,41\}$. The plots represent the choice of the first $j$ points, $j\in\{2,5,9\}$.}
	\label{fig:existence_choice}
\end{figure}  

We now choose points $x_j\in [0,L)$  inductively (see Figure \ref{fig:existence_choice}) and start with $x_1 = \alpha = \varepsilon$, which clearly satisfies \eqref{eq:thm_compact_cond3.3}. Assume now we have picked $x_1,\dots, x_{j-1}$. We choose $x_j$ to be \emph{the minimal element in $\alpha \Z$ that satisfies the Schoenberg-Whitney conditions \eqref{eq:thm_compact_cond3.3}}. Precisely, let
$$
    \ell = \max \Big \lbrace 0, \tilde \d^k_j -m+1+\sum\limits_{n=1}^{j-1} \left(1+\tilde \d^k_n\right)\Big\rbrace \, ,
$$
and define
\begin{equation} \label{eq:in12}
    x_{j} = \min \{ \alpha k: k\in \Z, \alpha k \in
    [\ell+\varepsilon, \ell+1-\varepsilon]\} \, . 
\end{equation}
As observed in  Step 1, we can always choose such $x_j$. The condition $x_j\in [\ell + \varepsilon , \ell +1-\varepsilon ]$ is clearly stronger than the Schoenberg-Whitney conditions \eqref{eq:thm_compact_cond3.3}.
The chosen points are in $[0,L)$ because 
$$
0\leq \ell\leq \tilde \d_N -m+1+\sum\limits_{n=1}^{N-1}
\left(1+\tilde \d _n\right) = L-1<L.
$$
 
\textbf{Step 3. Partitioning $\mu $ and $\R $.} We now partition the sequence $\mu $ into blocks $M_k := \{ \mu _{N_{k-1} + 1 } , \dots, \mu _{N_{k}}\}$. We set   $N_{-1}= 0$ and construct $N_k$ inductively. Given $N_{k-1}$, we choose  $N_k\in\Z$  minimal with 
$$
\sum_{n=N_{k-1} +1}^{N_k}  (1 + \mu _j) \geq L+m-1.
$$
We carry out the local step for each block and obtain points $x_{N_{k-1}+j}\in [kL, (k+1)L)$ with submultiplicities $\tilde \d_{N_{k-1}+j}$ satisfy all conditions of Theorem \ref{thm:main_compact}. 
Consequently, the set $\{ x_m : m\in \Z \} = \bigcup _{k\in \Z} \{x_{N_{k-1} + j}: j=1 , \dots , N_{k}\} $ with multiplicities $(\mu_{m})_{m\in\Z}$ form a sampling set in $V^p(\varphi )$. 
\footnote{Note that only at the interval ends
$N_k$ we had to reduce the multiplicities to apply Theorem~\ref{thm:main_compact}.}

\textbf{Step 4. Density.} Since we have 
$$
\sum _{n=N_{k-1}+1} ^{N_{k}}
(1+\mu _{n}) = L+m-1 +(\mu_{N_k} - \tilde \mu_{N_k}) \leq L+2m-2
$$
samples in every interval $[kL,(k+1)L)$,
the weighted Beurling density is bounded by $D^-(X,\mu ) \leq
\frac{L+2m-2}{L} = 1+ \frac{2m-2}{L} < 1+\nu $.
\end{proof}

\begin{remark} 
Note that a controlled perturbation of $X$ by setting $\tilde{x_j} = x_j + \gamma _j , j\in \Z$ for $|\gamma_j|\leq \tau \varepsilon$ for a fixed $\tau\in(0,1)$ still satisfies the conditions of Theorem~\ref{thm:main_compact}. Additionally, there exist sampling sets with multiplicities $(X,\mu_X)$ which are sampling for $V^p(\varphi)$, $D^-(X,\mu_X)<1+\nu$, but $D^-(X)<1$, in particular, $X$ is not sampling for $V^p(\varphi)$. A simple example is a sequence that alternates between $s_1$ and $s_2$, $s_1+s_2>0$. If $L\in\N$  satisfies $ L+m-1 = T(s_1+s_2+2)$ for sufficiently large  $T\in\N$, then $\tilde\d_j = \mu_j$ for all $j\in\Z$, there are $2T$ sampling locations in each interval $[kL,(k+1)L)$ and the resulting set $X$ is $L$-periodic. Its lower Beurling density is given by 
\begin{align}
    D^-(X) & = \lim \limits_{K\to\infty} \frac{\# (X\cap [-kL, kL])}{2kL} = \frac{2T}{L} = \frac{2T}{T(s_1+s_2+2)-(m-1)} \\
    &= \Big(\frac{s_1+s_2}{2}+1 -\frac{m-1}{2T}\Big)^{-1} <1,
\end{align}
whenever $T > \frac{m-1}{s_1+s_2}$. This argument can be  generalized
to more complicated sequences $\mu$.
\end{remark}

\section{Weak limits}\label{sec:weak_limits}

A particular case to consider is sampling with a constant multiplicity $S$  for some integer $0\leq S\leq m-1$. 
Since every point comes with $S+1$ data, one expects that a suitably modified gap condition might be sufficient for sampling with derivatives. Based on Corollary \ref{cor:max_gap_OV}, it is natural to conjecture that
\begin{equation}\label{eq:max_gap_const}
\frac{\mg(X)}{1+S}<1
\end{equation}
is sufficient for sampling. This is indeed correct, but it does not follow from Theorem \ref{thm:main_compact}. One can show that there are sets that satisfy \eqref{eq:max_gap_const} but violate the conditions \eqref{eq:thm_compact_cond3.1}, \eqref{eq:thm_compact_cond3.2}, and \eqref{eq:thm_compact_cond3.3}. These examples  will be addressed in Section \ref{sec:discussion}. 

To study maximal gap conditions and achieve an optimal result, we will use the  technique of weak limits. These go back to Beurling~\cite{BeurlingCollected1989} and have been introduced for the investigation of \sis s in \cite{GroechenigEtAl2017, GroechenigEtAl2019}. 

We recall  Beurling's notion of a weak limit of a sequence of sets. A sequence $\left(X_n\right)_{n\in\N}$ of subsets of $\R$ \textit{converges weakly} to a set $Y\subseteq\R$, denoted by  $X_n\overset{w}{\longrightarrow} Y$, if for every open bounded interval $(a,b)$ and every $\varepsilon>0$, there exists an $n_0\in\N$ such that for all $n\geq n_0$,
\begin{equation}
X_n \cap (a,b)\subseteq Y+(-\varepsilon, \varepsilon)\qquad\text{and}\qquad Y\cap (a,b)\subseteq X_n+(-\varepsilon,\varepsilon).
\end{equation}
Given a tuple of sets $\overrightarrow{X}= \left( X^0,\dots, X^S\right)$, a tuple $\overrightarrow{Y}=\left(Y^0,\dots,Y^S\right)$ is called a \emph{weak limit of integer translates} of $\overrightarrow{X}$ if there exists a sequence $(k_n)_{n\in\N}\subseteq\Z$ such that $X^s-k_n\overset{w}{\longrightarrow}Y^s$ for all $0\leq s\leq S$. An important part of the definition is the fact that the sequence of translates is the same for all $Y^s$.
The set of weak limits of integer translates of $X$ is denoted by  $W_\Z(\overrightarrow{X})$. 
Weak limits enter sampling theory through the  following characterization of sampling sets \cite[Thm.~2.1]{GroechenigEtAl2019}.
\begin{theorem}\label{thm:sampling_weak_limits}
Assume that $\Phi= \left( \Phi^0, \dots, \Phi^S \right)$ consists of continuous functions with a minimal decay property $|\Phi ^s(x)| \leq C(1+|x|)^{-1-\varepsilon}$ and $\Phi $ has stable integer shifts. Let $\overrightarrow{X} = \left(X^0,\dots, X^S\right)$ be a tuple of separated sets. Then the following are equivalent. 
\begin{enumerate}[(i)]
\item $\overrightarrow{X}$ is a sampling set for $V^p(\Phi)$ for some $p\in[1,\infty]$.
\item $\overrightarrow{X}$ is a sampling set for $V^p(\Phi)$ for all $p\in[1,\infty]$.
\item Every weak limit  $\overrightarrow{Y}\in W_\Z(\overrightarrow{X})$ is a sampling set for $V^\infty (\Phi)$.
\item Every weak limit  $\overrightarrow{Y}\in W_\Z(\overrightarrow{X})$ is a uniqueness set for $V^\infty(\Phi)$.
\end{enumerate}
\end{theorem}
When sampling on $(X,\mu_X)$, this theorem is applied to the generator $\Phi = (\varphi, \varphi',\dots,\varphi^{(S)})$ and the sets
\begin{equation}
X^{s}:=\lbrace x\in X : \mu_X(x)\geq s\rbrace,\qquad 0\leq s\leq \max\left\lbrace \mu_X(x) : x\in X\right\rbrace.
\end{equation} 
In this case, $\overrightarrow{X}-k_n $ can be identified with $(X,\mu_X)-k_n\coloneqq (X-k_n, \mu_{X-k_n})$, where
\begin{equation}\label{eq:mu_of_shift}
\mu_{X-k_n}(x-k_n) = \mu_X(x),\qquad x\in X.
\end{equation}
For separated sets, we can use an alternative characterization of weak convergence (cf. \cite[Lem.~4.4.]{GroechenigEtAl2015}, \cite[Prop.~3.2.]{GroechenigEtAl2019}).
\begin{proposition}\label{prop:weak_limit_separated}
Let $(X,\mu_X)$ be a separated set with multiplicity function $\mu_X:X\to \{0,\dots, S\}$, $(Y,\mu_Y)$ be a set with multiplicity function $\mu_Y$, and $\left(k_n\right)_{n\in\N}\subseteq\Z$ a sequence of integers. Then $X^s-k_n\overset{w}{\longrightarrow} Y^s$ as $n\to \infty$ for all
$0\leq s\leq S$, if and only if
\begin{equation}
\sum\limits_{x\in X}\left(1+\mu_X(x)\right)\delta_{x-k_n} \longrightarrow \sum\limits_{y\in Y}\left(1+\mu_Y(y)\right)\delta_y,\qquad\text{ as }n\longrightarrow\infty,
\end{equation}
in the $(C_c^*,C_c)$-topology (where $C_c$ denotes the class of continuous functions with compact support).
\end{proposition}
For splines, an additional difficulty arises that is not formally covered by Theorem \ref{thm:sampling_weak_limits}. The last derivative $\varphi^{(m-1)}$ is only piecewise continuous with finitely many jump discontinuities. In this case, we have to impose an additional condition on the $m$-th sampling set. The modified version of Theorem \ref{thm:sampling_weak_limits} goes as follows:
\begin{corollary}\label{cor:sampling_weak_mult}
Let $\varphi\in C^{S-1}(\R)$ be a compactly supported, piecewise $C^{S}$-function with finitely many jump discontinuities $\mathcal{J}\subseteq \R$. Further, assume that $\varphi$ has stable integer translates. Let $X$ be a separated set with multiplicity function $\mu_X :X\to \lbrace 0,\dots S\rbrace$ satisfying
\begin{equation}\label{eq:dist_requirement}
    \mathrm{dist}\big(\lbrace x\in X : \mu_X(x) = S\rbrace \,,\mathcal{J}+\Z\big)>0.
\end{equation}
Then the following statements are equivalent: 
\begin{enumerate}[(i)]
\item $(X,\mu_X)$ is a sampling set for $V^p(\varphi)$ for some $p\in[1,\infty]$.
\item $(X,\mu_X)$ is a sampling set for $V^p(\varphi)$ for all $p\in[1,\infty]$.
\item Every weak limit  $(Y, \mu_Y)\in W_\Z(X,\mu_X)$ is a sampling set for $V^\infty (\varphi)$.
\item Every weak limit  $(Y,\mu_Y)\in W_\Z(X,\mu_X)$ is a uniqueness set for $V^\infty(\varphi)$.
\end{enumerate}
\end{corollary}
\begin{remark}
The equivalence of \textit{(i)} and \textit{(ii)} holds without condition \eqref{eq:dist_requirement}. For PEB-splines of order $m$, \eqref{eq:dist_requirement} is only needed if $X^{m-1} = \lbrace x\in X : \mu(x) = m-1\rbrace$ is non-empty, i.e., when $(m-1)$-th derivatives are part of the data. If $X^{m-1}=\emptyset$, then \eqref{eq:dist_requirement} is void and not required.

The proof of Corollary \ref{cor:sampling_weak_mult} is a small generalization of \cite[Thm.~3.1]{GroechenigEtAl2017}. For the sake of completeness, we sketch the modifications in the  
\hyperref[sec:appendix]{Appendix}.
\end{remark}

When  we  label  a separated set $X\subseteq\R$ as a strictly increasing sequence, we can define the weighted maximum gap as 
\begin{align}
    \mg(X,\mu_X)& \coloneqq \max\bigg\lbrace
    \sup\limits_{j\in\Z}\frac{x_{j+1}-x_{j}}{1+\mu_X(x_j)},\,
    \sup\limits_{j\in\Z}\frac{x_{j+1}-x_{j}}{1+\mu_X(x_{j+1})}\bigg\rbrace \notag \\
    & =  \sup\limits_{j\in\Z}\,\frac{x_{j+1}-x_{j}}{1+\min\{\mu_X(x_j), \mu_X(x_{j+1})\}} \, .
\label{c3}
\end{align}
We  will show that $\mg(X,\mu_X)<1$ 
is sufficient for a set to be a sampling set.

To apply Corollary \ref{cor:sampling_weak_mult}, we first  study  the weighted maximum gap
of a weak limit $(Y,\mu_Y)\in W_\Z(X,\mu_X)$. 

\begin{lemma}\label{lemma:MG_weak_preserved}
Let $(X,\mu_X)$ be a separated set with a multiplicity function $\mu_X$ and $(Y,\mu_Y)$ be a  weak limit of $(X,\mu_X)$. Then 
\begin{equation}
\mg(Y,\mu_Y) \leq \mg(X,\mu_X).
\end{equation}
\end{lemma}
\begin{proof}
Let $(Y,\mu _Y)$ by a weak limit of $(X,\mu _X)$. Since $X$ is separated, we can use Proposition \ref{prop:weak_limit_separated} to describe the
weak convergence by means of measures 
\begin{equation} \label{c5}
    \begin{split}
    \sum\limits_{x\in X}\left(1+\mu_X(x)\right)\delta_{x-k_n} \longrightarrow \sum\limits_{y\in Y} \left(1+\mu_Y(y)\right)\delta_y, \qquad \text{ as} n\longrightarrow\infty. 
    \end{split}
\end{equation}

For arbitrary $\varepsilon>0$, we  choose a pair $y_j<y_{j+1}$ of consecutive points in $Y$ such that 
\begin{equation}
    \frac{y_{j+1}-y_{j}}{1+\min\{\mu_Y(y_j), \mu_Y(y_{j+1})\}} >\mg(Y,\mu_Y)-\varepsilon.
  \end{equation}
Let $\delta >0$ be the separation constant of $X$ and $f_0,\,f_1\in C_c(\R)$ be two auxiliary functions such that $0\leq f_0,f_1\leq 1$, 
\begin{equation}
    \begin{split}
    \supp(f_0)\subseteq \left(-\tfrac{\delta}{2}, \tfrac{\delta}{2}\right),&\qquad f_0(x) = 1\ \text{for all }x\in \left(-\tfrac{\delta}{4},\tfrac{\delta}{4}\right), \\ \supp(f_1)\subseteq \left(y_j, y_{j+1}\right), &\qquad f_1(x) = 1\ \text{for all } x\in \left(y_j+\tfrac{\delta}{2},y_{j+1}-\tfrac{\delta}{2}\right).
    \end{split}
\end{equation}
Then by \eqref{c5} 
\begin{align*}
    \lefteqn{\sum\limits_{x\in X}\left(1+\mu_X(x)\right)f_0(x-k_n-y_j) =}\\
    &= \sum\limits_{x\in X} \left(1+\mu_X(x)\right) \delta_{x-k_n}(T_{y_j}f_0) \to \sum\limits_{y\in Y}\left(1+\mu_Y(y)\right)\delta_y
    (T_{y_j}f_0) = 1+\mu_Y(y_j).
\end{align*}

Since $X$ is $\delta$-separated, $f_0\geq 0$ and $ \supp(f_0)\subseteq \left(-\frac{\delta}{2}, \frac{\delta}{2}\right)$, for every $n\in\N$ there is at most one $x\in X$, say $x_{n_j}$, such that
\begin{equation}
    |x_{n_j}-k_n-y_j|<\tfrac{\delta}{2}.
\end{equation}
So $x_{n_j}-k_n \to y_j$ as $n$ tends to infinity and consequently $\mu_X(x_{n_j}) = \mu_Y(y_j)$ for all $n\in\N$. Likewise, there is a sequence $\widetilde{x}_{n_j}-k_n - y_{j+1}$ such that 
\begin{equation}\label{c6}
    |\widetilde{x}_{n_j}-k_n-y_{j+1}|<\tfrac{\delta}{2}
\end{equation}
and $\mu_X(\widetilde{x}_{n_j}) = \mu(y_{j+1})$ for all $j\in\N$ sufficiently large.

Using the weak*-convergence for $f_1$, we obtain
\begin{equation}
    \sum\limits_{x\in X} \left(1+\mu_X(x)\right) f_1(x-k_n) \longrightarrow \sum\limits_{y\in Y}\left(1+\mu_Y(y)\right)f_1(y) = 0,
\end{equation}
because $\supp(f_1)\cap Y = \emptyset$. As $f_1\geq 0$, this implies that for all sufficiently large $n\in\N$, $\supp(f_1)\cap (X-k_n ) \subseteq [y_j, y_{j}+\frac{\delta}{2})\cap [ y_{j+1}-\frac{\delta}{2}, y_{j+1})$.

By \eqref{c6} and the $\delta$-separability of the sets, these must be the points $x_{n_j}-k_n$ and
$\widetilde{x}_{n_j}-k_n$, so that necessarily  $\widetilde{x}_{n_j} = x_{n_j+1}$. Consequently, since $\mu_{X-k_n} =\mu_X(\cdot +k_n)$,
\begin{equation}
    \begin{split}
    \mg(X,\mu_X) 
    \geq &\, \frac{x_{n_j+1}-x_{n_j}}{1+\min\{\mu_{X}(x_{n_j}), \mu_{X}(x_{n_{j+1}})\}} \\
    \geq &\,
    \frac{(x_{n_j+1}-k_n)-(x_{n_j}-k_n)}{1+\min\{\mu_{X-k_n}(x_{n_j}-k_n),
    \mu_{X-k_n}(x_{n_j+1}-k_n)\}} \\[.5ex]      
    \geq &\,
    \frac{y_{j+1}-y_{j}}{1+\min\{\mu_Y(y_j), \mu_Y(y_{j+1})\}} \\[.5ex]
    > & \, \mg(Y,\mu_Y)-\varepsilon.
    \end{split}
\end{equation}
Since $\varepsilon>0$ was arbitrary, the claim follows.
\end{proof}

\section{Sampling via maximal gap and  weak limits}\label{sec:main2}

In this section, we will prove two sampling theorems that use gap conditions for the sampling set. 
These conditions  are  easier to check than the technical conditions of Theorem~\ref{thm:main_compact}. 

We first relate the weighted maximum gap to the Schoenberg-Whitney conditions~\eqref{eq:x_pos_condition}. The following  statement does not make reference to splines and only involves arrangements of points on the line. 

\begin{lemma}[Combinatorial lemma]\label{lem:maxgap_SW_conditions}
Let $m\geq 2$ and let $(X,\mu_X)$ be a separated set with multiplicity function $\mu_X$. Assume that  $\mg(X,\mu_X)<1$. Then for all $x \in X$ there exist  $N,L\in\N$, $N,L\geq 2$, depending on $x$, and points 
$$
    x^*_1 < x^*_2<\dots <x^*_{N} \text{ in } X\cap \big[\left\lfloor x \right\rfloor ,\left\lfloor x \right\rfloor+L\big]
$$
with submultiplicities $\mu _j\in \N$, 
\begin{equation} \label{eq:ii1}
    \d_j \leq \mu_X(x^*_j),   \qquad \qquad  1\leq j\leq N \, ,  
\end{equation}
satisfying the following properties:
\begin{enumerate}
\item[(i)]  dimension count: 
\begin{equation} \label{eq:hmhm3}
    \sum\limits_{j=1}^{N}(1+\d_j) = L+m-1,    
\end{equation}
\item[(ii)] Schoenberg-Whitney conditions:
\begin{equation}
    \label{eq:c7}
    x_j^* \in  ( \d_j-m+1, 1 ) + \left\lfloor x \right\rfloor  + \sum\limits_{n=1}^{j-1} (1+\d_n), \qquad 1\leq j\leq N
\end{equation}
with the usual adaptation if $\d_j = m-1$,
\item[(iii)]  
$ X\cap \big[\left\lfloor x \right\rfloor  +L-1,\left\lfloor x \right\rfloor+L\big] \neq \emptyset \, .$
\end{enumerate}
$(^*)$ If $x^*_{N} = \left\lfloor x \right\rfloor+L\in\Z$, then, in addition, we have  $\mu_{N}\leq m-2$.  
\end{lemma}
Although the number of points $N$ and the length of the interval $L$ depend on $x$, they may be bounded uniformly by 
\begin{equation} 
    N \leq  1+\left\lceil (m+1)(1-\mg(X,\mu_X))^{-1} \right\rceil = N_0,\qquad \text{ and } L \leq m(N_0+m+2).
\end{equation}
\begin{proof}
The proof is split into several steps. 
We make a first  selection of  points $x_1^*,\dots, x_{K}^*$ that satisfy all conditions. In the generic case, we are done, but in some exceptional cases, we need to adapt the intermediate parameters and add more points.  

\textbf{Reduction.} 
Since the problem is invariant under integer shifts and reindexing, we may  assume without loss of generality that
$\left\lfloor x \right\rfloor=0$ and $x_1 = \min \{x\in X : x\geq 0\}$. We need to find $N$, $L\in\N$, points $x^*\in [0,L]^N$ and submultiplicities $\d\in \{ 0,\dots,m-1\}^N$ with the stated properties.       
   
\textbf{Step 1: (Preliminary) choice of $N$ and $L$.} 
The condition $\mg (X,\mu _X) <1$ leads to a bound for consecutive samples $x_{j+1} - x_j \leq \mg (X, \mu _X) (1+\mu_X(x_j))$ for all $j\in \Z $. By telescoping we obtain~\footnotemark \footnotetext{Here we use $\frac{x_{j+1}-x_{j}}{1+\mu_X(x_{j})} \leq \mg(X,\mu_X) $.}
\begin{equation} \label{eq:fi1}
    x_n-x_1 = \sum _{j=1}^{n-1} (x_{j+1}-x_j) \leq \mgg \sum_{j=1}^{n-1} (1+\mu _X(x_j)) < \sum_{j=1}^{n-1} (1+\mu _X(x_j)) \, .
\end{equation}
Thus the number of data $\sum_{j=1}^{n-1} (1+\mu _X(x_j))$ always exceeds the length of the sampling interval $x_n-x_1$ by fixed factor $\mgg ^{-1} >1$. This leaves some room for manipulation.

In particular, for $n$ large enough, $(n-1)(1-\mgg ) > m+1$, we obtain 
\begin{align}    
    \left\lceil x_{n}\right\rceil +m -1 &\leq x_n-x_1 + 2 + m-1 \notag\\ 
    &\leq \mgg \sum   _{j=1}^{n-1} (1+\mu _X(x_j)) + m+1 \notag\\
    &\leq \mgg \sum   _{j=1}^{n-1} (1+\mu _X(x_j)) + (1 - \mgg ) (n-1)\notag \\
    &\leq \sum   _{j=1}^{n-1} (1+\mu _X(x_j)) \, . \label{eq:fi2}
\end{align}
Now we define $K$ as the smallest integer with     \begin{equation}\label{eq:N_choice}
    \left\lceil x_{K}\right\rceil +m -1 \leq \sum\limits_{j=1}^{K}(1+\mu_X(x_j)). 
\end{equation}
Clearly $K\geq 1$ and by the preceding argument $K \leq (m+1)/(1-\mgg)$. This choice of $K$ implies the reverse inequality
\begin{equation} \label{eq:fi3}
    \left\lceil x_{j}\right\rceil +m -1 >
    \sum\limits_{k=1}^{j}(1+\mu_X(x_k)) \qquad \text{ for } j=1, \dots , K-1 \, . 
\end{equation}
The inequalities \eqref{eq:N_choice} and \eqref{eq:fi3} contain the key idea of the proof.

As the preliminary choice for the interval length we set
$$
L_0 = \left\lceil x_{K}\right\rceil \, . 
$$
The final choice of the number of points $N$ will be from $\{2, 3,K,K+1\}$ and $L \leq  \left\lceil x_{K}\right\rceil+ m\leq (K+1)m + m$, whence we obtain their boundedness independent of the initial choice of $x\in X$. 

\textbf{Step 2: Choice of the initial $x^*$ and their submultiplicities $\d$ when $K>1$.}
Assume first that $K> 1$. The exceptional case $K=1$ will be treated later. We choose 
$$
x_j^*\coloneqq x_j,\quad \ \d_j \coloneqq \mu_X(x_j),\qquad \
1\leq j \leq K-1 \, ,
$$
and first show the Schoenberg-Whitney conditions \eqref{eq:c7} for these points. Since  $\left\lfloor x_1\right\rfloor = 0$, for $j=1$ these are the inequalities $x_1<1$ and $ x_1 \geq 0 \geq \mu_1 -m+1$ with equality if and only if $\mu_1 = m-1$.

For all $j=  2,\dots, K-1$, \eqref{eq:fi1} yields the upper inequality
\begin{equation}\label{eq:c_lower_bound}
    \begin{split}
    x_j^*
    & = x_j = x_j-x_1 + x_1 \\ 
    & \leq \mg(X,\mu_X) \sum\limits_{n=1}^{j-1}(1+\mu_X(x_n)) + 1
    <1 + \sum\limits_{n=1}^{j-1}(1+\mu_n) \, .
    \end{split}
\end{equation}
The lower bound is contained in~\eqref{eq:fi3} via     
$$
    x_j >   \left\lceil x_{j}\right\rceil -1  > -m + \sum\limits_{n=1}^{j}(1+\d_n)   
    = 1+\d_{j} -m + \sum\limits_{n=1}^{j-1}(1+\d_n).
$$
\textbf{Step 3.}
For $j=K$ we set $x_K^*=x_K$ and  
$$
    \d_{K}\coloneqq \left\lceil x_{K}\right\rceil + m-2-\sum\limits_{j=1}^{K-1}(1+\mu_j).
$$
This choice ensures condition~\eqref{eq:hmhm3};      
\begin{equation}\label{eq:c11}
    \sum\limits_{j=1}^{K}(1+\d_j) = \left\lceil x_{K}\right\rceil
    +m-1 = L_0+m-1 \, .
\end{equation}
To check the range of $\d_K$, we use   \eqref{eq:N_choice} and obtain 
\begin{align}
    1+ \d_{K}  &=  \left\lceil x_{K}\right\rceil +
    m-1-\sum\limits_{j=1}^{K-1}(1+\mu_j)\\
    & \leq  \sum\limits_{j=1}^{K}(1+\mu_j) -
    \sum\limits_{j=1}^{K-1}(1+\d_j)  = 1+\mu_X(x_K) \, , 
\end{align}
and by \eqref{eq:fi3} for $j=K-1$
\begin{align*}
    1+ \d _K &= \lceil x_{K}\rceil + m-1-\sum\limits_{j=1}^{K-1}(1+\mu_j)\\
    & \geq \lceil x_{K-1}\rceil +
    m-1-\sum\limits_{j=1}^{K-1}(1+\mu_j) >0 \, .
\end{align*}
Since $\d _K \in \Z$, we obtain $0\leq \d_K \leq \mu _X(x_K^*)$. 
To show that $x_{K}^*$ satisfies the Schoenberg-Whitney condition \eqref{eq:c7}, we recall that for $j\leq K-1$, $\d_j = \mu_X(x_j)$, thus the estimate in \eqref{eq:c_lower_bound} holds for $x_{K}^*$ as well. The lower bound holds by choice of $\d_{K}$: 
\begin{align*}
    x_K^* = x_{K} > \left\lceil x_{K}\right\rceil -1 = \d_{K} -m+1 +\sum\limits_{n=1}^{K-1} \left(1+\d_n\right).
\end{align*}

\textbf{Step 4: Generic Case  $K>1, x_K>1$. } If $x_{K}>1$, and additionally $x_{K-1}\notin\Z$ or $\d_{K-1} \leq m-2$, then we set $N \coloneqq K$ and $L\coloneqq L_0= \lceil x_{K}\rceil$ so that $L-1 \leq x_N^* = x_K \leq L$. By construction  the points $x^*_1 < \dots <x^*_{N} $ with submultiplicities $\d_1,\dots, \d_N $ satisfy (i) - (iii) and we are done. 

\vspace{3mm}

It remains to consider the exceptional cases: (i) $x_K\geq 1$ with $x_K \in \Z$ and $\d _K = m-1 $,  (ii) $x_K < 1$, and  (iii) $K=1$. 

\textbf{Step 5. Case $x_{K}^* = x_{K}\geq 1$, $x_{K}\in\Z$, $\d_{K}= m-1$}. In this case, we add the next  sampling point $x_{K+1}$ and set
$$
    x_{K+1}^* \coloneqq x_{K+1} \quad \text{ and }\quad \d_{K+1} \coloneqq \left\lceil x_{K+1}\right\rceil - x_{K}-1 \, .
$$
Since $\left\lceil x_{K+1}\right\rceil - x_k >0$ is an integer, we have $\d _{K+1} \geq 0$.
We set $N = K+1$ and $L = \left\lceil x_{K+1}\right\rceil$ and verify items (i) --- (iii) after adding the new point $x^*_{K+1}$. We substitute \eqref{eq:c11} and the definition of $\d _K$ and obtain 
\begin{align}
    \sum\limits_{j=1}^{K+1}(1+\d_j) = \lceil x_{K} \rceil +m-1 +1+\d_{K+1} = \left\lceil x_{K+1}\right\rceil+m-1 = L+m-1. 
\end{align}
Next, 
\begin{align}
    \d_{K+1} &= \left\lceil x_{K+1}\right\rceil - \left\lceil x_{K}\right\rceil -1 < x_{K+1} -  x_{K}  < \mu_{X}(x_{K+1})+1.~\footnotemark 
\end{align}
\footnotetext{Here we use $\frac{x_{j+1}-x_{j}}{1+\mu_X(x_{j+1})} \leq \mg(X,\mu_X) <1$.}         
This  inequality of integers implies  $\d_{K+1} \leq \mu_{X}(x_{K+1})$. If also $x_{K+1} \in \Z $, then
$$
\d_{K+1} =x_{K+1}-  x_{K} -1 < \mu _X(x_{K+1}) 
$$
and thus $ \d _{K+1} \leq m-2$, which takes care of the extra condition~$(^*)$.

To verify the Schoenberg-Whitney conditions \eqref{eq:c7}, we use 
$\sum\limits_{n=1}^{K} (1+\d_n) = \lceil x_K \rceil + m-1$ from \eqref{eq:c11} and find the upper estimate from 
\begin{align*}
    x_{K+1}^* &= x_{K+1} - x_K + \lceil x_K \rceil \\
    &< 1+ \mu _X(x_K) + \sum\limits_{n=1}^{K} (1+\d_n)  - m+ 1 \leq 1+ \sum\limits_{n=1}^{K}  (1+\d_n) \,  ,
\end{align*}
whereas the lower estimate follows from
\begin{align*}
    \d _{K+1} - m+1 + \sum\limits_{n=1}^{K} (1+\d_n) &= \left\lceil x_{K+1}\right\rceil - \left\lceil x_{K}\right\rceil -1 -m+1 + \lceil x_K \rceil + m-1 \\
    &= \lceil x_{K+1} \rceil -1 < x_{K+1} \, .
\end{align*}
    
The last condition (iii) is satisfied by definition because $x_{K+1}^* \in [L,L+1)$.

\textbf{Step 6: Case $x_K <1, K>1$.} 
Let $\widetilde{K}$ be the minimal index with $ x_{\widetilde{K}}> 1$. Due to its minimality and the maximum gap condition, the points 
$ x_{\widetilde{K}-1}, x_{\widetilde{K}}$ satisfy
\begin{equation}\label{eq:c8}
    x_{K} \leq x_{\widetilde{K}-1} \leq 1 \quad \text{ and }\quad 1 <  x_{\widetilde{K}}< x_{\widetilde{K}-1} +m\leq 1+m.
\end{equation}
We set  $N\coloneqq K+1$, $L\coloneqq \big\lceil x_{\widetilde{K}} \big\rceil$, and 
$$ 
    x_{K+1}^* \coloneqq x_{\widetilde{K}},\qquad \d_{K+1}  \coloneqq \left\lceil x_{\widetilde{K}} \right\rceil-2< x_{\tilde{K}} -1 \, .
$$ 
This choice guarantees that  $L=\lceil x_{\widetilde{K}} \rceil  \geq 2$, $\d_{K+1}\geq 0$, and $x_{K+1}^* \in [L-1,L] \cap X$.
We have      
\begin{equation}\label{eq:c13}
    \d_{K+1} < x_{\widetilde{K}} -1 \leq x_{\widetilde{K}} - x_{\widetilde{K}-1}  < 1+\mu_X(x_{\widetilde{K}})~\footnotemark,
\end{equation}
and again $\d_{K+1}\leq \mu_X(x_{K+1}^*)$. 
\footnotetext{Here we use $\frac{x_{j+1}-x_{j}}{1+\mu_X(x_{j+1})} \leq \mg(X,\mu_X) <1$.} 
     
For the Schoenberg-Whitney condition \eqref{eq:c7} for $x_{K+1}$, we observe that \eqref{eq:c11} says that $\sum\limits_{j=1}^{K}(1+\d_j) = \left\lceil x_{K} \right\rceil +m-1 = m$, so that \eqref{eq:c8} says that   
$$
x_{K+1}^* = x_{\tilde{K}} < 1+m = 1+\sum\limits_{j=1}^{K}(1+\d_j) \, ,
$$ 
and
$$
\d _{K+1} - m+1 + \sum\limits_{j=1}^{K}(1+\d_j) = \lceil x_{\tilde{K}} \rceil  -2 - m+1 + m = \lceil x_{K+1}^* \rceil  -1 < x_{K+1}^*   \, .
$$ 

\textbf{Step 7: Case $K=1$.}
In this case, \eqref{eq:fi3} is void and cannot be used for   the Schoenberg-Whitney conditions  for $j>1$.  We need to modify the selection of points. We distinguish three subcases. 

If $x_1>0$, then the inequalities $\mu _X(x_1) - m + 1 \leq 0 < x_1 <1$ express the Schoenberg-Whitney condition for $x_1$. We proceed as in Step~6 and choose $x_2^* = x_{\tilde{K}}$ with $x_{\tilde{K}-1} \leq 1 < x_{\tilde{K}}$ and $\d _2 = \lceil x_{\tilde{K}} \rceil - 2$. We set $N=2$ and $L= \lceil x_{\tilde{K}} \rceil \geq
2$. Conditions (i) --- (iii) are verified readily as before.

If $x_1=0$, then we must have  $\mu _X(x_1) = m-1$.  If $x_2 \leq 1$, then we set 
\begin{align*}
    x_2^* &\coloneqq x_2,\qquad \d_2\coloneqq 0 \, , \\
    x_3^* & = x_{\tilde{K}} \qquad \d_3 = \lceil x_{\tilde{K}} \rceil - 2 \, .
\end{align*}
With $N=3$ points and interval length $L= \lceil x_{\tilde{K}} \rceil \geq 2$, all conditions (i) --- (iii) are satisfied.

If $x_1=0$ and $x_2 >1$, set $N=2$ $N\coloneqq 2$, $L\coloneqq \left\lceil x_{2} \right\rceil \geq 2$, and 
$$
    x_{2}^* \coloneqq x_{2},\qquad \d_{2}  \coloneqq \left\lceil x_{2} \right\rceil-2.
$$
Properties (i) --- (iii) are  verified as in Step~6. 

As we have covered all possible cases, this  concludes the proof.
\end{proof}

\begin{theorem}\label{thm:combinatorial_USet} 
Let $\varphi$ be a PEB-spline of order $m\in\N$ and $X\subseteq\R$ be a separated set with a multiplicity function $\mu_X:X\to\{0,\dots, m-1\}$. If 
\begin{equation}
\mg(X,\mu_X)=\max\left\lbrace  \sup\limits_{j\in\Z} \frac{x_{j+1}-x_{j}}{1+\mu_X(x_j)}\, ,\, \sup\limits_{j\in\Z} \frac{x_{j+1}-x_{j}}{1+\mu_X(x_{j+1})}\right\rbrace <1,
\end{equation}
then $(X, \mu)$ is a uniqueness set for $V^p(\varphi)$ for all $1\leq p\leq \infty$.
\end{theorem}

\begin{proof}
Assume that $f^{(s)}(x) = 0$ for all $0\leq s\leq \mu_X(x)$ for all $x\in X$. We need to prove that $f\equiv 0$.

\textbf{Step 1.} 
Let $x\in X$ be arbitrary and set $a = \left\lfloor x\right\rfloor$. By Lemma~\ref{lem:maxgap_SW_conditions}, there exists an integer length $L\geq 2$ and points $a\leq x_1^* <\dots <x_N^*\leq a+L$ with multiplicities $\d_1,\dots,\d_N\in\{0,\dots, m-1\}$  satisfying properties (i) --- (iii) (correct dimension count, Schoenberg-Whitney conditions and $x_N^* \in [L-1,L]$). By assumption on $f$, we know that 
\begin{equation}
    0 = f^{(s_j)}(x^*_j) =  \Big(\sum\limits_{\ell\in\Z} c_\ell T_\ell \varphi\Big)^{(s_j)}(x^*_j) = \Big(\sum\limits_{\ell=a- m+1}^{a+ L-1} c_\ell T_\ell \varphi\Big)^{(s_j)}(x^*_j), \quad 0\leq s_j\leq \d_j,\ 1\leq j\leq N.
\end{equation}
Since the Schoenberg-Whitney conditions are satisfied for the points $x_j^*$ with multiplicities $\d _j, j=1, \dots ,N$, Corollary \ref{cor:matrix_multiplicities} asserts that  $f=0$ on $[a, a+L]$. 

\textbf{Step 2.} By induction we construct a sequence of intervals  $[a_k,a_k+L_k]$ such that 
\begin{align*}
    \bigcup\limits_{k=1}^\infty \, [a_k,a_k+L_k] &= [\left\lfloor x\right\rfloor,\infty), \\
    [a_k+L_k-1,a_k+L_k] \cap X &\neq \emptyset \\
    f &\equiv 0 \text{ on } [a_k,a_k+L_k]    \text{ for all } k\in\N \, .
\end{align*}
We start with $a_1 = \lfloor x \rfloor $ and $L_1=L$ from Step~1. Assume that we have already constructed $K$ intervals with this
property by induction. By the induction hypothesis, there exists a point 
\begin{equation}\label{eq:c10}
    \tilde x\in [a_K+L_K-1,a_K+L_K]\cap X \, .
\end{equation}
Set $a_{K+1}\coloneqq \left\lfloor \tilde x\right\rfloor$. By Step~1, there exists an $L_{K+1}\geq 2$ such that $f=0$ on the interval $[a_{K+1},a_{K+1}+L_{K+1}]$ and 
$[a_{K+1}+L_{K+1}-1,a_{K+1}+L_{K+1}] \cap X\neq \emptyset.$
The sequence $a_K + L_K$ is unbounded, because $a_{K+1} + L_{K+1}  \geq L_K + a_K-1 + 2$. Consequently, 
\begin{equation}
    \bigcup\limits_{k=1}^K [a_k,a_k+L_k] =[\left\lfloor x\right\rfloor,a_{K+1}+L_{K+1}].
\end{equation}    
Since $X$ is unbounded from below and $a_K + L_K$ is unbounded, $f$ vanishes on the union $\bigcup _{x\in X, K \to \infty}  [ \lceil x \rceil, a_K + L_K] = \R$. This shows $f\equiv 0$.
\end{proof}

We conclude the section with the proof of the sampling theorem.

\begin{theorem}[Maximum Gap Theorem] \label{thm:max_gap}
Let $\varphi$ be a PEB-spline of order $m\geq 2$
and let $X\subseteq\R$ be a separated set with multiplicity function $\mu_X:X\to \lbrace 0,\dots, m-1\}$. If the multiplicity function satisfies 
\begin{equation}
    \mathrm{dist}\left(\left\lbrace x \in X: \mu_X(x) = m-1\right\rbrace,\ \Z\,\right) >0
\end{equation}
and the weighted maximum gap of $(X,\mu_X)$ satisfies
\begin{equation}
\mg(X,\mu_X)<1,
\end{equation}
then $(X, \mu_X)$ is a sampling set for $V^p(\varphi)$, $1\leq p\leq\infty$.
\end{theorem}
\begin{proof}
By Corollary \ref{cor:sampling_weak_mult}, it suffices to prove that every weak limit of integer translates is a uniqueness set for $V^\infty(\varphi)$. Let $(Y,\mu_Y)\in W_\Z(X,\mu_X)$ be a weak limit of integer translates of $(X,\mu_X)$. Lemma {\ref{lemma:MG_weak_preserved}} asserts that $\mg(Y,\mu_Y)<1$. 
By Theorem \ref{thm:combinatorial_USet}, $(Y,\mu_Y)$ is a uniqueness set for $V^\infty(\varphi)$, and we are done.
\end{proof}

\textbf{Constant multiplicity.} 
Assume that $\mu _X \equiv s < m-1$, i.e., we sample the same number of derivatives at every point $x\in X$. In this case,
$$
    \mg (X,\mu _X)  = \frac{\mg (X)}{s+1} \, .
$$
We then obtain the following consequence.

\begin{corollary} \label{allequal}
Let $\varphi$ be a PEB-spline of order $m\geq 2$ and let $X\subseteq\R$ be a separated set. Assume that the multiplicity function is constant $\mu _X(x) = s < m-1$ for all $x\in X$. If the  maximum gap of $X$ satisfies 
\begin{equation}
\mg(X)< s+1,
\end{equation}
then $(X, \mu_X)$ is a sampling set for $V^p(\varphi)$, $1\leq p\leq\infty$.

In particular, if $\mu _X \equiv 0$ and $\mg (X) <1$, then $X$ is a sampling set for $V^p(\varphi), 1\leq p\leq \infty $. 
\end{corollary}

For $B$-splines and multiplicity  $\mu_X\equiv 0$, this sampling theorem was already proved in \cite{AldroubiGroechenig2000}. The method with weak limits yields a different proof.  

For  uniform sampling on $X=\alpha \Z$, the maximum gap is $\mg(\alpha\Z) = \alpha$, and $\alpha < s+1$ guarantees the sampling inequality
$$
\sum _{k\in \Z} \sum _{j=0}^s |f^{(j)}(\alpha k)|^p \asymp \|f\|_p^p \qquad \text{ for all } f\in V^p(\varphi ) \, .
$$
Note that the weighted density in this case is $D^-(\alpha \Z, \mu) = \frac{s+1}{\alpha}>1$. As $\alpha $ approaches $s+1$, we obtain sampling sets of density arbitrarily close to the critical
density. The condition $\alpha <s+1$ cannot be improved. The set $\alpha \Z$ for $\alpha >s+1$ fails to have the necessary density and is thus not a sampling set~\footnote{ One can show that
  there is always some $x_0\in \R$, such that 
  $x_0+ (s+1) \Z $ fails to be sampling.}.
In this case, Theorem~\ref{thm:max_gap} is optimal and cannot be improved. 

\section{Discussion of the methods}\label{sec:discussion}

\textbf{Comparison.} Both theorems (Theorems~\ref{thm:main_compact} and~\ref{thm:max_gap}) are based on similar ideas. We choose a suitable covering of $\R $ and
study sampling locally on each interval of the partition by means of the active collocation matrix.

In Theorem~\ref{thm:main_compact} we control the distance of a sampling point to the
boundary of the support of $\varphi $ as  a security margin. This parameter guarantees uniform bounds for the inversion of the local collocation matrices (Step~2 in the proof).

In Theorem~\ref{thm:max_gap}, we omit that condition, and the Schoenberg-Whitney conditions imply only the invertibility of the local collocation matrices. As a consequence, we obtain only uniqueness sets. In this case, the theory of weak limits and subtle criteria for sampling sets help to derive stability estimates.

Both main theorems are optimal in some sense, but their assumptions are not comparable. To see this, let us consider a few simple examples. In all examples, the generator $\varphi $ is a PEB-spline of order $m$ with support in $[0,m]$. 

\begin{example}
Let $X=\alpha \Z$ and consider $\mu _X(\alpha j) = \tfrac{1}{2} (1 + (-1)^j)$ so that $\mu _X(\alpha j) = 1$ for even $j$ and $\mu _X(\alpha j) = 0$ for odd $j$. The weighted maximum gap is then
$ \mg (X,\mu _X)  = \alpha \, .$ Consequently, by Theorem~\ref{thm:max_gap} $\alpha <1$ implies the sampling inequality
$$
    \sum _{j\in \Z } |f(\alpha j)|^p + \sum _{j\in \Z} |f'(2\alpha j)|^p \asymp \|f\|_p^p \qquad \text{ for all } f\in V^p(\varphi ) \, .
$$
However, in an interval of length $K$ there are $K/(2\alpha )$ samples with multiplicity $0$ and $K/(2\alpha )$ samples with multiplicity $1$ (up to a bounded error term), so the weighted Beurling density of $(X,\mu _X)$ is
$$  
    D^-(X , \mu _X) = \frac{3}{2\alpha} >\frac{3}{2}\, ,
$$
i.e., the maximum gap does not deliver an optimal estimate. 

By contrast, by Theorem  \ref{thm:existence}, for every $\varepsilon >0$, using the conditions of Theorem~\ref{thm:main_compact}, we can produce a set $(X,\mu _X)$  with weighted Beurling density $D^-(X,\mu _X) < 1+\varepsilon $ that is sampling. Hereby, $\mu_X$ can be chosen to be an alternating function as above, and the explicit construction in the proof of Theorem \ref{thm:existence} provides a subset of a lattice.
\end{example}

\begin{example}\label{ex:gap_vs_compact}
We provide an example of a set uniformly bounded away from $\Z$ with a constant multiplicity function $\mu\equiv S-1>0$ and maximum gap $\mg(X) = S\,\mg(X,\mu_X)<S$ that  does not satisfy the assumptions of Theorem \ref{thm:main_compact}.
While the combinatorial theorem relied heavily on \emph{overlaps}, Theorem \ref{thm:main_compact} asks for a \emph{partition} in intervals $I(k)=I_{M,L}(k)$. The last assumption \eqref{eq:thm_compact_cond3.3} implies that there exists a sampling point, namely $x_1^k \in X\cap I_{M,L}(k)$, in the first subinterval of $I_{M,L}(k)$ of length $1$;  
\begin{equation}
\begin{split}
    & [M+kL,M+kL+L) \cap [0-m+M+kL+1+\varepsilon, M+kL+1-\varepsilon) \\
    =& [M+kL, M+kL+1-\varepsilon)\subseteq [M+kL, M+kL+1) \eqqcolon J_{M,L}(k).
\end{split}
\end{equation}
Note that $M, L$ are supposed to be chosen globally. Furthermore, for any fixed $L$, one can always assume that $0\leq M<L$, as $I_{M+L,L}(k) = I_{M,L}(k+1)$. Hence, to prove that the claimed set exists, it suffices to construct a separated set $X\subseteq \R$ with maximum gap $\mg(X)<2$ and $\mathrm{dist}(X,\Z)>0$ such that for all $M\in \Z,\, L\in\N$, there exists an integer $k\in\Z$ with $X\cap J_{M, L}(k)=\emptyset$. 
The set can be constructed inductively by starting with a sufficiently dense lattice $X_0$, e.g., $X_0 = \frac{1}{n}\Z$ for some  $n\in\N , n\geq 2$, and for each $L\in\N$ and each $0\leq M\leq L$ removing the first $n$ points of $X_0$ in the interval $I_{M,L}((M+L)2^{2^L})$ (these are the points $X\cap J_{M,L}(k)$), resulting in a new sampling set $X$ given by 
$$
    X\coloneqq X_0 \setminus \bigg( \bigcup\limits_{L\in\N} \bigcup\limits_{M = 0}^{L-1} X\cap J_{M,L}((M+L)2^{2^L})\bigg).
$$
One can verify that each of the distinguished  intervals is sufficiently far apart from the others so that the maximum gap is $\mg(X) = 1+\frac{1}{n}<2$. 
By construction of the set, there are no integers $M, L \in \Z$ that provide a suitable partition of the real axis where Theorem \ref{thm:main_compact} is applicable. However, this is also a sampling set because the weighted maximum gap is $\mg(X,\mu_X) = \frac{1}{S}\mg(X) <\frac{2}{S}\leq 1$, so the conditions of Theorem \ref{thm:max_gap} are satisfied.
\end{example}

\section{Implications for Gabor systems} \label{sec:gabor}
In the final section, we exploit the connection between sampling in \sis s and Gabor frames and  improve  a result in~\cite{KloosStoeckler2014}.

We denote with $\pi(x,\omega)$, $(x,\omega)\in\R^2$, the time-frequency shift (operator) acting on functions as $\pi(x,\omega)f(t) = e^{2\pi i \omega t}f(t-x)$. In terms of stable expansions, given a \textit{window function} $\varphi$ and a discrete set $\Lambda\subseteq\R^{2}$, one asks  when a \textit{Gabor system} $\G(\varphi,\Lambda)$, defined as 
\begin{equation}
    \G(g, \Lambda) = \lbrace \pi(\lambda)\varphi : \lambda\in\Lambda\rbrace,
\end{equation}
is a \textit{(Gabor) frame}, i.e., it satisfies the frame inequality
\begin{equation}
    A\norm{f}_2^2\leq \sum\limits_{\lambda\in\Lambda}\abs{\left\langle f, \pi(\lambda) \varphi\right\rangle}^2\leq  B\norm{f}_2^2, \qquad f\in L^2(\R),
\end{equation}
with frame bounds $0 < A \leq B<\infty$ independent of $f\in L^2(\R)$. From the general frame theory, the inequality implies a reconstruction formula
\begin{equation}
    f = \sum\limits_{\lambda\in\Lambda} \langle f, \pi(\lambda)\varphi\rangle\, \psi_\lambda= \sum\limits_{\lambda\in\Lambda} \langle f, \psi_\lambda\rangle\, \pi(\lambda)\varphi, \qquad f\in L^2(\R)
\end{equation}
with 
$\left(\langle f, \pi(\lambda)\varphi\rangle\right)_{\lambda\in\Lambda}$, $\left(\langle f, \psi_\lambda\rangle\right)_{\lambda\in\Lambda}\in \ell^2(\Lambda)$ and $\psi_\lambda\in L^2(\R)$ for all $\lambda\in\Lambda$. 
Therefore, given $\varphi$, we are interested in determining those sets $\Lambda\subseteq\R^2$ whose Gabor system $\G(\varphi,\Lambda)$ is a frame.

While there is no manageable tool available for arbitrary point configurations, semi-regular sets of type $\Lambda = X\times\Z$ can be treated by the  connection between sampling in shift-invariant spaces and Gabor frames~\cite[Thm.~3.1, Thm.~3.3]{GroechenigEtAl2017}.
\begin{theorem}\label{thm:Gabor_vs_Sampling} 
Assume that $\varphi\in C(\R)$ decays as $|\varphi(x)|\lesssim (1+|x|)^{-1-\varepsilon}$ for an $\varepsilon>0$, and has stable integer shifts. Let $X\subseteq\R$ be a separated set. Then the following are equivalent:
\begin{enumerate}[(i)]
\item The family $\G(\varphi, (-X)\times \Z)$ is a frame for $L^2(\R)$.
\item $X$ is a sampling set of $V^p(\varphi)$ for some $p\in[1,\infty]$.
\item $X$ is a sampling set of $V^p(\varphi)$ for all $p\in[1,\infty]$.
\end{enumerate}
\end{theorem}
The following corollary is a small extension of a result in \cite{AldroubiGroechenig2000} from B-splines to PEB-splines.
\begin{corollary}\label{cor:PEB_Gabor}
Let $\varphi$ be a PEB-spline of order $m\geq 2$. Assume $X\subseteq\R$ is a separated set satisfying the following property:

$(^*)$ There exist integers $M, L>0$ and $\varepsilon >0$, such that for every $k\in\Z$, there exist points $x_1^k< x_2^k<\dots <x^k_{L+{m-1}}$ in $X\cap I_{M,L}(k)$ with
\begin{equation} 
x_j^k\in [M+kL+j-m+\varepsilon, M+kL+j-\varepsilon].
\end{equation}
Then $\G(\varphi, (-X)\times\Z)$ is a Gabor frame.
In particular, if $X$ is separated and $\mg(X)<1$, then $\G(\varphi, X\times\Z)$ is a Gabor frame. 
\end{corollary} 
\begin{proof}
We proved that $\varphi$ has stable integer shifts in the proof of Theorem \ref{thm:main_compact}. 
The sampling property of $X$ follows from Theorem \ref{thm:main_compact}, applied to $(X, 0)$, and Corollary \ref{cor:max_gap_OV}. The claim follows from Theorem \ref{thm:Gabor_vs_Sampling}. 
\end{proof}

A direct consequence of the last corollary is the lattice case.
\begin{corollary}  
Let $\varphi$ be a PEB-spline of order $m\geq 2$. Then   $\G(\varphi, \alpha\Z\times\Z)$ is a Gabor frame if and only if $0 <\alpha<1$. 
\end{corollary}
\begin{proof} 
The sufficient part is due to Corollary \ref{cor:PEB_Gabor}. The necessary part is due to the Balian-Low theorem, see, e.g., ~\cite{ BenedettoEtAl1995}.
\end{proof} 

\section{Appendix}\label{sec:appendix} 
We sketch the proof of (an extended version of)  Corollary \ref{cor:sampling_weak_mult}. The only changes are required for the last, right-continuous $\varphi^{(s)}$. 
\begin{corollary}
Let $\varphi:\R\to \C$ be a compactly supported, piecewise continuous function with finitely many jump discontinuities $\mathcal{J}\subseteq \R$. Further, assume that $\varphi$ has stable integer translates. Let $X$ be a separated set satisfying
\begin{equation}
    \varepsilon \coloneqq \mathrm{dist} \big(X\,,\mathcal{J}+\Z\big)>0.
\end{equation}
Then the following statements are equivalent: 
\begin{enumerate}[(i)]
\item $X$ is a sampling set for $V^p(\varphi)$ for some $p\in[1,\infty]$.
\item $X$ is a sampling set for $V^p(\varphi)$ for all $p\in[1,\infty]$.
\item Every weak limit  $Y\in W_\Z(X)$ is a sampling set for $V^\infty (\varphi)$.
\item Every weak limit  $Y\in W_\Z(X)$ is a uniqueness set for $V^\infty(\varphi)$.
\end{enumerate}
\end{corollary}
\begin{proof}[Proof sketch]
The sampling property can be  expressed in terms of the pre-Gramian matrices
\begin{equation}
    P_X(\varphi)\coloneqq \left(\varphi(x-k)\right)_{x\in X,k\in \Z}.
\end{equation}
Precisely, $X$ is sampling for $V^p(\varphi )$ if and only if $P_X(\varphi )$ is bounded above and below on $\ell ^p(\Z )$. 
By Schur's test, $P_X:\ell^p(\Z)\to\ell^p(X)$ is a bounded linear operator for all $p\in[1,\infty]$. We make use of a non-commutative version of Wiener's Lemma \cite[Prop.~8.1]{GroechenigEtAl2015}.
\begin{proposition} 
Let $X, Y\subseteq \R$ be relatively separated and $A\in\C^{X\times Y}$ be a matrix such that ~\footnote{The decay conditions can be weakened, e.g., $\abs{A_{x,y}} \lesssim \Theta(x-y)$ for some continuous upper bound $\Theta$ in the Wiener-amalgam space $ W(\R)$.} 
\begin{equation}
    \abs{A_{x,y}} \lesssim (1+|x-y|)^{-1-\varepsilon}\qquad x\in X,y\in Y\qquad \text{for some }\varepsilon>0.
\end{equation}
The operator $A$ is bounded below on some $\ell^{p_0}(Y)$, $1\leq p_0\leq \infty$, i.e., $\norm{Ac}_{p_0}\geq C_0\norm{c}_{p_0}$ for all $c\in\ell^{p_o}(Y)$, if and only if $A$ is bounded below on \emph{all} $\ell^p(Y)$, $1\leq p\leq \infty$.
\end{proposition} 
Since $\varphi$ is supported on some compact set $K\subseteq\R$ and bounded,
\begin{equation}
    \big|\big( P_X(\varphi)\big)_{x,k} \big| \lesssim (1+|x-k|)^{-1-\varepsilon}\qquad x\in X,k\in \Z \qquad \text{for all }\varepsilon>0.
\end{equation}
This gives the equivalence $(i)\Leftrightarrow (ii)$. 

$(ii)\Rightarrow (iii)$ Let $Y\in W_\Z(X)$, i.e., $X-k_n\to Y$ weakly for some $(k_n)_{n\in\N}\subseteq \Z$. Since $X$ is a sampling set, $P_X(\varphi):\ell^\infty(\Z)\to \ell^\infty(X)$ is bounded from below, so its pre-adjoint $P_X(g)':\ell^1(X)\to \ell^1(\Z)$, given by
\begin{equation}
    P_X(\varphi)' = \left(\varphi(x-k)\right)_{k\in\Z,x\in X},
\end{equation}
is onto by the closed graph theorem.
Fix now $c\in\ell^1(\Z)$. For every sequence $(c_{k-k_n})_{k\in\Z}$ there exists a sequence $a^n\in\ell^1(X)$ with $\norm{a^n}_1\lesssim \norm{c}_1$ and 
\begin{equation}
    c_{k-k_n} = \sum\limits_{x\in X} a^n_x g(x-k),\qquad k\in\Z.
\end{equation}
The weak convergence can be rewritten by means of distributions. Set $b^n_{x-k_n} = a^n_x$ and $\mu_n\coloneqq \sum_{z\in X-k_n} b^n_z\, \delta_z$.
Then $\mu_n \in C_0^*(\R)$ has support in $X-k_n$ and satisfies
\begin{equation}
    \norm{\mu_n}_{C^*_0(\R)} = \norm{b^n}_1 = \norm{a^n}_1\lesssim \norm{c}_1.
\end{equation}
By restricting to a subsequence and using the $w^*$-completeness of $C^*_0(\R)$, $\mu_n$ converges to $\mu\in C^*_0(\R)$ in $\sigma(C_0^*(\R),C_0(\R))$. 
Since $X-k_n\to Y$, the weak limit $Y$ is relatively separated and $\supp(\mu)\subseteq Y$ (see \cite[Lem.~4.3 and 4.5]{GroechenigEtAl2015}). Furthermore, $\mu =\sum_{y\in Y}b_y \delta_y$ for suitable $b\in\ell^1(Y)$, 
\begin{equation}
    \norm{b}_1 = \norm{\mu} _{C^*_0(\R)}\lesssim \liminf\limits_{n\to\infty} \norm{\mu_n}_{C^*_0(\R)}.
\end{equation}
Here comes the decisive modification. 
By assumption, $\mathrm{dist}(X-k_n,\mathcal{J}+\Z) = \mathrm{dist}(X,\mathcal{J}+\Z) >0$ for all $n\in\Z$ and therefore $\mathrm{dist}(Y,\mathcal{J}+\Z)>0$. 
Now we replace the generator $\varphi$ (which has discontinuities $\mathcal{J}$) by a continuous compactly supported function $\varphi_0$ such that $\varphi = \varphi_0$ on $\R\setminus \big(\mathcal{J}+\Z+(-\varepsilon , \varepsilon )\big)$ for an $\varepsilon \in (0, \mathrm{dist}(X,\mathcal{J}+\Z))$. 
In particular, $\varphi = \varphi_0$ on $Y$ and on $X-k_n$ for all $n\in\N$. Hence we obtain
\begin{align}
    \sum\limits_{y\in Y} b_y \varphi(y-k) & = \int_\R \varphi_0(x-k) \,d\mu(x) = \lim\limits_{n\to\infty} \int_\R \varphi_0(x-k) \,d\mu_n(x)  \\
    & = \lim\limits_{n\to\infty} \sum\limits_{z\in X-k_n} b_z^n \varphi_0(z-k)= \lim\limits_{n\to\infty} \sum\limits_{z\in X-k_n} b_z^n \varphi(z-k) = c_k.
\end{align}

$(iii)\Rightarrow (iv)$ is trivial. To obtain $(iv)\Rightarrow (ii)$, one proceeds precisely as in \cite{GroechenigEtAl2015}, with the analogous modification with $\varphi _0$ when it comes to the limits. 
\end{proof}
The lift to sampling with multiplicities in Corollary \ref{cor:sampling_weak_mult} is the same as in \cite{GroechenigEtAl2019}.

\def\cprime{$'$} \def\cprime{$'$} \def\cprime{$'$} \def\cprime{$'$}
  \def\cprime{$'$} \def\cprime{$'$}

\end{document}